\documentclass[12pt,reqno]{amsart}
\usepackage{amssymb,amsmath,amsthm}
\usepackage{graphicx}
\usepackage{subcaption}
\usepackage{fullpage}
\usepackage{scrextend}
\usepackage{siunitx}

\usepackage{enumerate}
\usepackage{tikz,calc}
\usepackage{tikz-cd}
\usetikzlibrary{automata,positioning}
\usepackage{comment}
\usetikzlibrary{calc}
\usepackage{epsfig,graphicx}
\usepackage{listings}
\usepackage{enumitem}

\usepackage{float}
\usepackage{bbm}
\usepackage[colorlinks=true, pdfstartview=FitH, linkcolor=blue, citecolor=blue, urlcolor=blue]{hyperref}

\newcommand{\brax}[1]{\left( #1 \right)}
\newcommand{\abs}[1]{\left|#1\right|}
\newcommand{\sqrax}[1]{\left[ #1 \right]}
\newcommand{\set}[1]{\left\{#1\right\}}
\newcommand{\ve}{\varepsilon}

\DeclareMathOperator{\circc}{Circ}
\DeclareMathOperator{\sym}{Sym}
\DeclareMathOperator{\refl}{refl}
\DeclareMathOperator{\fold}{fold}
\DeclareMathOperator{\vol}{Vol}

\newcommand{\N}{\mathbb{N}}
\newcommand{\R}{\mathbb{R}}

\newcommand{\Mod}[1]{\ (\text{mod}\ #1)}
\renewcommand{\Mod}[1]{{\ifmmode\text{\rm\ (mod~$#1$)}\else\discretionary{}{}{\hbox{ }}\rm(mod~$#1$)\fi}}

\newtheorem{theorem}{Theorem}[section]
\newtheorem{prop}[theorem]{Proposition}
\newtheorem{lemma}[theorem]{Lemma}

\theoremstyle{definition} 

\numberwithin{theorem}{section}

\allowdisplaybreaks

\begin{document}

\title{On Axial Symmetry in Convex Bodies}

\author{Ritesh Goenka}
\address{Department of Mathematics \\ University of British Columbia \\ Room 121, 1984 Mathematics Road \\ Vancouver, BC, Canada V6T 1Z2}
\email{rgoenka@math.ubc.ca}
\author{Kenneth Moore}
\address{Department of Mathematics \\ University of British Columbia \\ Room 121, 1984 Mathematics Road \\ Vancouver, BC, Canada V6T 1Z2}
\email{kjmoore@math.ubc.ca}
\author{Wen Rui Sun}
\address{Department of Mathematical and Statistical Sciences \\ University of Alberta \\ Edmonton AB, Canada T6G 2G1}
\email{wrsun@ualberta.ca}
\author{Ethan Patrick White}
\address{Department of Mathematics \\ University of British Columbia \\ Room 121, 1984 Mathematics Road \\ Vancouver, BC, Canada V6T 1Z2}
\email{epwhite@math.ubc.ca}
\thanks{The fourth-named author is supported in part by an NSERC PDF}

\begin{abstract}
    For a two-dimensional convex body, the Kovner-Besicovitch measure of symmetry is defined as the volume ratio of the largest centrally symmetric body contained inside the body to the original body. A classical result states that the Kovner-Besicovitch measure is at least $2/3$ for every convex body and equals $2/3$ for triangles. Lassak showed that an alternative measure of symmetry, i.e., symmetry about a line (axiality) has a value of at least $2/3$ for every convex body. However, the smallest known value of the axiality of a convex body is around $0.81584$, achieved by a convex quadrilateral. We show that every plane convex body has axiality at least $\frac{2}{41}(10 + 3 \sqrt{2}) \approx 0.69476$, thereby establishing a separation with the central symmetry measure. Moreover, we find a family of convex quadrilaterals with axiality approaching $\frac{1}{3}(\sqrt{2}+1) \approx 0.80474$. We also establish improved bounds for a ``folding" measure of axial symmetry for plane convex bodies. Finally, we establish improved bounds for a generalization of axiality to high-dimensional convex bodies.
\end{abstract}

\maketitle

\section{Introduction}
\label{sec:intro}

Symmetry is central to the study of mathematics. Common examples include the invariance of lattices and tilings under translation, rotation, and reflection. In the context of geometric objects, there are several different notions of symmetry. For example, a three-dimensional body is said to exhibit mirror symmetry if there is a plane, reflecting the body through which yields itself. Symmetry measures are often used to quantify symmetry in convex bodies (compact convex sets). We refer the reader to \cite{Gru} and \cite{GT} (see also \cite{deV}) for detailed accounts of various quantitative symmetry measures arising in convex geometry.

For $n \in \N$ and $0 \le k < n$ integer, Chakerian and Stein \cite{CS} define the \emph{$k$-symmetry} of a convex body $K \subset \R^n$ to be
\begin{equation*}
    \sym_k(K)=\frac{1}{\vol_n(K)}\sqrax{\max_{k \text{ dim. flats } \mathcal{L}} \vol_n(K \cap \refl_\mathcal{L} (K))}\, ,
\end{equation*}
where $\vol_n(\cdot)$ is the $n$-dimensional Lebesgue measure and $\refl_\mathcal{L}(K)$ is the convex body obtained by reflecting $K$ through $\mathcal{L}$. Let
\begin{equation*}
    \sigma(n,k) = \inf_{K \subset \R^n\text{ convex body}} \sym_k(K)\, ,
\end{equation*}
be the size of the largest $k$-symmetric body inside of a convex body in $\R^n$, minimized over all unit volume convex bodies. Chakerian and Stein \cite[Theorem~3]{CS} also prove the general lower bound
\begin{equation}
\label{eqn:glb}
    \sigma(n, k) \ge \frac{\text{max}\{k!, (n-k)!\}}{2^{n-k} n!}\, .
\end{equation}

The case $k = 0$ corresponds to reflection through points. Thus $\sigma(n,0)$ is a measure of the lowest amount of central symmetry that an $n$-dimensional convex body has. The bound $\sigma(n,0) \ge 2^{-n}$ given by \eqref{eqn:glb} was proved previously by Stein \cite{Ste}. Improved bounds have been established in low dimensions, including $\sigma(2,0) = 2/3$ \cite{Bes,Far,LL} and $\sigma(3,0) \ge 2/9$ \cite{BR}. The former of these is known as the Kovner–Besicovitch theorem and is notable in that the bound is realized by triangles. F\'ary and R\'edei \cite{FR} computed the central symmetry of the regular $n$-simplex $\Delta^n$, thereby establishing the upper bound
\begin{equation}
\label{eqn:centup}
    \sigma(n,0) \le \sym_0(\Delta^n) = \frac{1}{(n+1)^n} \sum_{i=0}^{\lfloor n/2 \rfloor} (-1)^i {n+1 \choose i} (n+1-2i)^n\, .
\end{equation}

Similarly, the case $k = n-1$ corresponds to reflection through hyperplanes. Thus $\sigma(n,n-1)$ is a measure of the lowest amount of hyperplane mirror symmetry that an $n$-dimensional convex body has. A related measure of asymmetry called \emph{chirality coefficient} has been introduced in attempts to quantify chirality in chemistry \cite{Gil}. There are two differences: (1) it is defined as $1$ minus the symmetry measure since it is a measure of asymmetry, and (2) it also allows for rotation and translation of the reflected body to maximize the overlap volume. Gilat and Gordon \cite{GG} showed that in two dimensions, the addition of rotation and translation does not improve overlap. The same is true in three dimensions if the body of maximum overlap is unique. Gilat and Gordon also obtained upper bounds on the chirality coefficient using Ball's volume ratio inequality \cite{Ball} for John ellipsoids inscribed in convex bodies. In particular, the volume ratio inequality yields
\begin{equation*}
    \sigma(n,k) \ge \frac{\pi^{n/2} n!}{n^{n/2} (n+1)^{(n+1)/2} \Gamma(\frac{n}{2} + 1)}\, ,
\end{equation*}
since ellipsoids are $k$-symmetric for all $0 \le k < n$. The above estimate is weaker than \eqref{eqn:glb} for all sufficiently large values of $n$.

For plane convex bodies, the measure $\sym_1(\cdot)$ is often called axiality. Krakowski \cite{Kra} established a lower bound of $5/8$ on axiality, and Lassak \cite{Las} improved it to $2/3$. A better lower bound of $2 (\sqrt{2} - 1)$ was proved by Nohl \cite{Noh} for the axiality of centrally symmetric bodies. Buda and Mislow \cite{BM} showed that the same lower bound also holds for triangles. Both these results are tight in the sense that there is a centrally symmetric parallelogram and a sequence of triangles with axiality equal to/approaching $2(\sqrt{2} - 1)$. This stood as the best upper bound on $\sigma(2,1)$ until Choi improved the upper bound to approximately $0.81584$ and conjectured that the true value of $\sigma(2,1)$ is around $0.81$~\cite{Cho}. 

Building on the ideas of Lassak \cite{Las}, we obtain improved lower and upper bounds on $\sigma(2,1)$.

\begin{theorem}
\label{thm:axiality}
    Any planar convex body has axiality at least $\frac{2}{41} (10+3\sqrt{2})$. There is a sequence of quadrilaterals for which the axiality values approach $\frac{1}{3}(\sqrt{2}+1)$. Therefore, $0.69476 < \frac{2}{41} (10+3\sqrt{2}) \le \sigma(2,1) \le \frac{1}{3}(\sqrt{2}+1) < 0.80474$.
\end{theorem}

Theorem~\ref{thm:axiality} shows that there is a strict separation between $\sigma(2,1)$ and $\sigma(2,0)$. In particular, $\sigma(2,1) > \sigma(2,0)$. Setting $k = n-1$ in \eqref{eqn:glb} yields the lower bound
\begin{equation}
\label{eqn:axlb}
    \sigma(n,n-1) \ge \frac{1}{2n}\, ,
\end{equation}
which together with \eqref{eqn:centup} implies that $\sigma(n,n-1) > \sigma(n,0)$ for all $n \ge 11$. It remains to determine whether a similar separation also holds for intermediate dimensions $3 \le n \le 10$.

In high dimensions, no non-trivial upper bounds have been found for $\sigma(n,n-1)$. We prove the following general theorem, which implies that the upper bound for $\sigma(2,1)$ carries over to $\sigma(n,n-1)$ with higher values of $n$.

\begin{theorem}
\label{thm:higherdimensions}
Let $n \ge 2$ and $0 \le k < n$ be integers. We have 
\begin{equation*}
    \max\set{\brax{2-2^{-1/n}}^{-n}, \sigma(n,k)} \geq \sigma(n+1,k+1)\, .
\end{equation*} 
In particular, $\sigma(n,n-1) \le \sigma(2,1) \le \frac{1}{3}(\sqrt{2}+1)$.
\end{theorem}

The best known lower bound on $\sigma(n,n-1)$ is given by \eqref{eqn:axlb}, which decays to zero as $n \to \infty$. An interesting problem is to determine if $\sigma(n,n-1)$ remains bounded away from zero as $n \to \infty$.

Lassak \cite{Las} defined a ``folding" measure of hyperplane mirror symmetry $\sym_{\text{fold}}(K)$, which is defined as twice the volume ratio of the largest portion of $K$ that can be folded into $K$ by reflection about a hyperplane to the original body $K$. Let
\begin{equation*}
    \varphi(n) = \inf_{K \subset \R^n\text{ convex body}} \sym_{\text{fold}}(K)\, ,
\end{equation*}
be the minimum symmetry among all $n$-dimensional unit volume convex bodies. Lassak \cite[Theorem~3]{Las} proved that $\varphi(2) \ge 1/4$ but did not provide any upper bound. An upper bound of $1/2$ was shown \cite{MN}, but as we will discuss later, this result is incorrect. We improve the lower bound and repair the upper bound on $\varphi(2)$. And we also provide an improved lower bound for centrally symmetric bodies.

\begin{theorem}
\label{thm:fold}
    Any planar convex body has folding symmetry at least $3/8$. There is a sequence of parallelograms for which the folding symmetry values approach $1/\phi$, where $\phi$ is the golden ratio. Therefore, $3/8 \le \varphi(2) \le 1/\phi < 0.61804$. Furthermore, any centrally symmetric planar convex body has folding symmetry at least $4/9$. 
\end{theorem}

The proofs of the lower and upper bounds in Theorem~\ref{thm:axiality} are discussed in Section~\ref{sec:axiality}. The bound in high dimensions is proved in Section~\ref{sec:high}. And the bounds for folding symmetry are proved in Section~\ref{sec:fold}.

\section{Axiality in the plane}
\label{sec:axiality}

We prove lower and upper bounds on the lowest value of axiality for plane convex bodies.

\subsection{Lower Bound}
\label{subsec:ax-lb}

We say a nondegenerate hexagon $ABCDEF$ is \emph{axially regular} if (i) the line $CF$ is an axis of symmetry for the hexagon, (ii) segments $\overline{AB}$, $\overline{CF}$, and $\overline{DE}$ are parallel, and (iii) $2|\overline{AB}| = |\overline{CF}| = 2|\overline{DE}|$. Lassak shows that an axially regular hexagon can be inscribed in every plane convex body and that such hexagons have at least $2/3$ the area of the convex body~\cite[\S 3]{Las}. By gluing triangles to axially regular hexagons, and considering triangles that well-approximate a convex body, we achieve an improved lower bound on planar axial symmetry.

Our method is to write down a \hyperref[axialprogram]{convex program} that captures our new ideas. Let $K \subset \mathbb{R}^2$ be a convex body. The objective of the optimization will be to minimize the variable $\lambda$, representing the area ratio of an axially symmetric convex body contained in $K$. By the aforementioned result of Lassak, we can find an axially regular hexagon inscribed in $K$. Let the vertices of the hexagon be $A,B,C,D,E,F$ as shown in Figure~\ref{axialhex}. We normalize so that the inscribed hexagon has unit area. Extend the edges of the hexagon to form a six-pointed star, with vertices labeled $A',B',C',D',E',F'$ in counterclockwise order as shown. Without loss of generality, assume $|\overline{AA'}| \geq |\overline{BC'}|$. The main variables of the program will be the area of $K$ contained in the 6 points of the star. Let $a,b,c,d,e,f$ be these areas in counterclockwise order, beginning with $a$ as the area of $K$ contained in triangle $AFA'$. Let $t = a+b+c+d+e+f$ be the total area of $K$ outside of the hexagon $ABCDEF$.

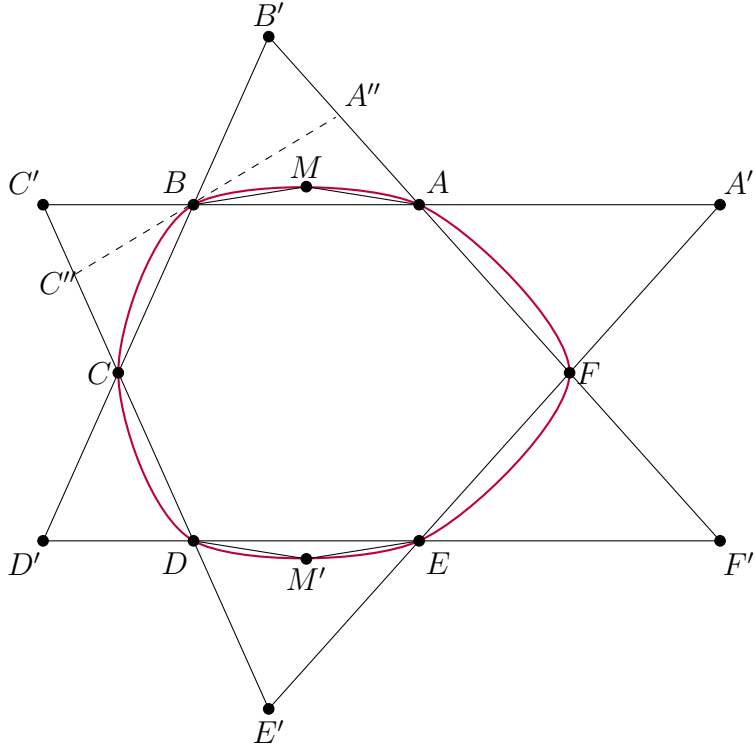
\begin{figure}[H]
\begin{center}
\begin{tikzpicture}[scale=1]

\def\x{3};
\def\s{0.5};
\def\h{2.23606};

\draw[thick, purple] plot [smooth cycle] coordinates {(-\x/2,\h) (\x/2,\h) (\x+\s,0) (\x/2,-\h) (-\x/2,-\h) (-\x+\s,0)};

\draw 
(-\s,-2*\h)--(-1.5*\x+2*\s,\h)
(-\s,-2*\h)--(1.5*\x+2*\s,\h)
(-\s,2*\h)--(-1.5*\x+2*\s,-\h)
(-\s,2*\h)--(1.5*\x+2*\s,-\h)
(1.5*\x+2*\s,-\h)--(-1.5*\x+2*\s,-\h)
(1.5*\x+2*\s,\h)--(-1.5*\x+2*\s,\h);

\filldraw
(\x/2,\h) circle(2pt)
(-\x/2,\h) circle(2pt)
(\x/2,-\h) circle(2pt)
(-\x/2,-\h) circle(2pt)
(-\x+\s,0) circle(2pt)
(\x+\s,0) circle(2pt)
(-1.5*\x+2*\s,\h) circle(2pt)
(-1.5*\x+2*\s,-\h) circle(2pt)
(1.5*\x+2*\s,\h) circle(2pt)
(1.5*\x+2*\s,-\h) circle(2pt)
(-\s,2*\h) circle(2pt)
(-\s,-2*\h) circle(2pt);

\draw
(\x/2+0.25,\h+0.3) node{$A$}
(-\x/2-0.25,\h+0.3) node{$B$}
(\x/2+0.25,-\h-0.3) node{$E$}
(-\x/2-0.25,-\h-0.3) node{$D$}
(-\x-0.25+\s,0) node{$C$}
(\x+\s+0.25,0) node{$F$}
(-1.5*\x+2*\s-0.25,\h+0.3) node{$C'$}
(-1.5*\x+2*\s-0.25,-\h-0.3) node{$D'$}
(1.5*\x+2*\s+0.25,\h+0.3) node{$A'$}
(1.5*\x+2*\s+0.25,-\h-0.3) node{$F'$}
(-\s,2*\h+0.3) node{$B'$}
(-\s,-2*\h-0.3) node{$E'$}
(0.75,3.7) node{$A''$}
(-3.3,1.2) node{$C''$};

\def\a{1.5725};
\def\b{0.925};
\draw[dashed]
(-\x/2-\a,\h-\b)--(-\x/2+\a+0.32,\h+\b+0.24);

\filldraw
(0,-\h-0.24) circle(2pt)
(0,\h+0.24) circle(2pt);

\draw
(0,-\h-0.24)--(\x/2,-\h)
(0,-\h-0.24)--(-\x/2,-\h)
(0,-\h-0.5) node{$M'$}
(0,\h+0.24)--(\x/2,\h)
(0,\h+0.24)--(-\x/2,\h)
(0,\h+0.5) node{$M$};

\end{tikzpicture}
\end{center}
\caption{An axially regular hexagon inscribed in the convex body $K$.}
\label{axialhex}
\end{figure}

Let the supporting line of $K$ at $B$ intersects $\overline{AB'}$ at $A''$ and $\overline{CC'}$ at $C''$. The sum of the areas $|\Delta AA''B|$ and $|\Delta BCC''|$ is at most $1/6$ since triangles $ABB'$ and $BCC'$ have areas at most $1/6$. This shows $b+c \leq 1/6$. Similarly, we have $d+e \leq 1/6$; this gives constraints \eqref{BCax} and \eqref{DEax}. The areas of $\Delta AFA'$ and $\Delta BCC'$ sum to $1/3$. We conclude 
\begin{equation*}
    a+b \leq |\Delta AFA'| \leq \frac{1}{3} - |\Delta BCC'| \leq \frac{1}{3} - d\, .
\end{equation*}
This shows $a+b+d \leq 1/3$, and by a similar argument $c+e+f \leq 1/3$; this verifies constraints \eqref{ABDax} and \eqref{CEFax}.

Now we add isosceles triangles to the hexagon while maintaining the axial symmetry about the line $CF$. Let the perpendicular bisector of $\overline{AB}$ meet the boundary of $K$ at $M$ and $M'$ as shown. By considering the supporting line of $K$ at $M$, we see that the area of triangle $ABM$ is at least $b/2$. Similarly, $\Delta DEM'$ has area at least $e/2$. The pair of isosceles triangles with bases $\overline{AB}$ and $\overline{DE}$, and height the minimum of the heights of triangles $ABM$ and $ABM'$, are contained in $K$ and symmetric about the line $CF$. It follows that $K$ contains an axially symmetric convex body of area at least $1 + \min\{b,e\}$. The idea of adding isosceles triangles can also be applied to the pairs of triangles $AFA',EFF'$ and $BCC',CDD'$. In summary, we can find an axially symmetric convex body of area at least
\begin{equation*}
    1 + \min\{a,f\}+ \min\{b,e\} + \min\{c,d\}\, ,
\end{equation*}
inside of $K$. This proves constraint \eqref{addax}.

If $K$ occupies a large portion of $\Delta A'C'E'$ or $\Delta B'D'F'$, then $K$ has almost as much axial symmetry as a triangle. We prove a short lemma to make this precise. 

\begin{lemma}
\label{subsetsym}
    Let $K \subseteq L \subset \mathbb{R}^2$ be convex bodies. For any line $\ell$, we have
    \begin{equation*}
        |K \cap \refl_\ell(K)| \geq |L \cap \refl_\ell(L)| - 2|L \setminus K|\, .    
    \end{equation*}
\end{lemma}

\begin{proof}
    Let $J = L \setminus K$. For notational ease, let $J^*$, $K^*$, and $L^*$ denote $\refl_\ell(J)$, $\refl_\ell(K)$, and $\refl_\ell(L)$, respectively. We have
    \begin{equation*}
        L \cap L^* = (K \cup J) \cap (K^* \cup J^*) = (K \cap K^*) \cup (K \cap J^*) \cup (J \cap (K^* \cup J^*))\, .
    \end{equation*}
    Since $|K \cap J^*|, |J \cap (K^* \cup J^*)| \leq |J|$ we have the desired result. 
\end{proof}

For any triangle $T \subset \mathbb{R}^2$, Buda and Mislow show that Sym$_1(T) \geq 2(\sqrt{2}-1)$~\cite{BM}. The area of triangle $A'C'E'$ is $3/2$, and so there is a line $\ell$ such that $|\Delta A'C'E' \cap \refl_\ell(\Delta A'C'E')| \geq 3(\sqrt{2}-1)$. The area inside $\Delta A'C'E'$ but outside of $K$ is at most $1/2 - a-c-e$. By Lemma~\ref{subsetsym}, this implies there is an axially symmetric convex subset of $K$ inside of triangle $A'C'E'$ of size at least
\begin{equation*}
    3(\sqrt{2}-1) - 2(1/2 - a-c-e) = 3\sqrt{2}-4 + 2(a+c+e)\, .
\end{equation*}
A similar argument applies to triangle $B'D'F'$, and this verifies constraints \eqref{triace} and \eqref{tribdf}.

\phantomsection
\label{axialprogram}
\begin{align}
\textsc{Axial symmetry } & \textsc{program.} \nonumber \\
\textsc{Variables: } & \lambda,\,a,\,b,\,c,\,d,\,e,\,f,\,t \nonumber \\
\textsc{Minimize: } & \lambda \nonumber \\
\textsc{Subject to: } & a,\,b,\,c,\,d,\,e,\,f \geq 0, \nonumber \\
& t = a+b+c+d+e+f\,, \nonumber \\
& b+c \leq 1/6\,, \label{BCax} \\
& d +e \leq 1/6\,, \label{DEax} \\
& a+ b+d \leq 1/3\,, \label{ABDax} \\
& c+e +f \leq 1/3\,, \label{CEFax} \\
& (1+t) \lambda \geq 1 + \min\{a,f\} + \min\{b,e\} + \min\{c,d\}\,, \label{addax} \\
& (1+t) \lambda \geq  3\sqrt{2} - 4 + 2(a+c+e)\,, \label{triace} \\
& (1+t) \lambda \geq 3 \sqrt{2}-4 + 2(b+d+f)\,. \label{tribdf}
\end{align}

\vspace{10pt}
\begin{prop}
\label{prop:ax-lower}
    $\sigma(2, 1) \ge \frac{2}{41}(10 + 3 \sqrt{2})$.
\end{prop}

\begin{proof}
    We solve the above \hyperref[axialprogram]{program} and prove the correctness of our solution in Appendix~\ref{axialanalysis}. Our idea is to treat the program as a linear program parametrized by $t$, and then determine the optimal objective value for the dual linear program in terms of $t$.
\end{proof}

\subsection{Upper bound}

In \cite{BM}, Choi improved the upper bound on $\sigma(2,1)$ to approximately 0.81584 using a computer assisted proof, providing the first example with axiality lower than $2\sqrt{2}-2$. He also conjectured that the true value of $\sigma(2,1)$ is around $0.81$. 

We implemented a basic algorithm to iteratively search for low symmetry shapes, using simulated annealing. The program is available on GitHub \cite{program}, and can be used to compute the axiality of polygons, including Choi's example, as well as run our search for lower symmetry ones. We were able to run this search to find new polygons with axiality $< 0.806$, one of which is pre-loaded as one launches the program.

We initially intended to present a computational proof of a lower symmetry body in a similar fashion to Choi. However, we later discovered a family of polygons which has lower axiality than anything we obtained experimentally. The family becomes flat as it approaches the minimum axiality, which means the program would never obtain this number, unless there is a different better construction.

\begin{prop}
\label{prop:r2part2}
    $\sigma(2,1)\leq \frac{1}{3}(\sqrt{2} + 1) \approx 0.8047378541$.
\end{prop}

\begin{proof}
    We consider the quadrilateral $\set{(0,0), (1,0), \left(1, \frac{\sqrt{2}}{1 + \sqrt{2}} \ve\right), \left(\frac{1}{\sqrt{2}},\ve\right)}$, and show that by sending $\ve\to 0$, its axiality approaches the desired number.

    \begin{figure}[H]
        \centering
        \includegraphics[scale=0.54]{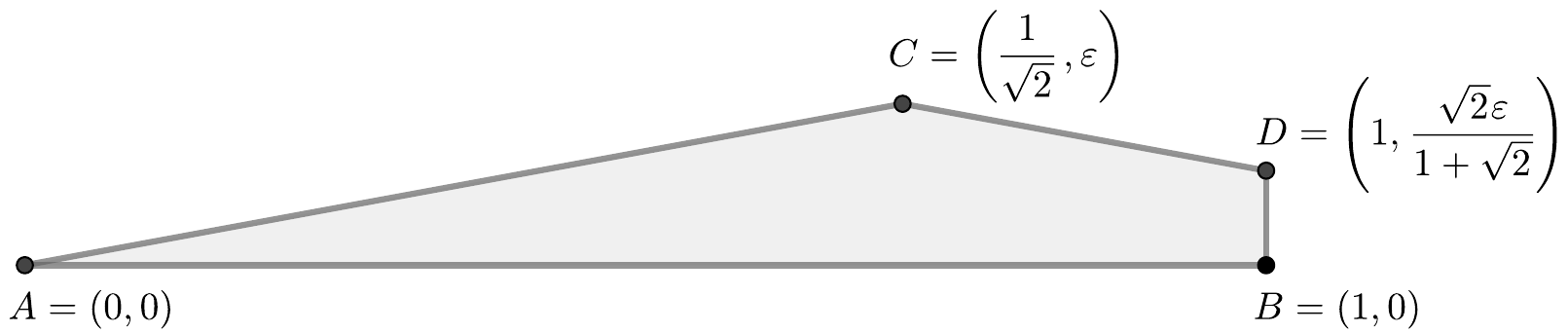}
        \caption{A family of quadrilaterals parametrized by $\ve$, which yields a new upper bound for $\sigma(2,1)$.}
        \label{fig:2Dupperbound}
    \end{figure}

    To do this, one must compute the area of overlap in every combination of angle and translation of the reflecting line. The computation is cumbersome, but routine. We have placed our analysis in Appendix \ref{app:axcomp}. In short, one finds that reflections in the angle bisector of $\angle CAB$ and the vertical line $x = 2/3$ both either attain or approach this number, and every other reflection line yields smaller overlap. 
\end{proof}

Proposition~\ref{prop:r2part2} together with Proposition~\ref{prop:ax-lower} yields Theorem~\ref{thm:axiality}.

\section{Folding symmetry in the plane}
\label{sec:fold}

Lassak defined the folding symmetry of a convex body in \cite{Las}. It is an alternative measure of hyperplane symmetry. Define the \emph{folding symmetry} of $K$ by
\begin{equation*}
    \sym_{\fold}(K) = \frac{2}{\vol_n(K)}\sqrax{\max_{\text{Halfspaces } \mathcal{H}} \set{\vol_n(K \cap \mathcal{H}): \refl_{\partial \mathcal{H}}(K \cap \mathcal{H}) \subset K}}\, .
\end{equation*}
We are interested in the minimal symmetry possible among all bodies, so define
\begin{equation*}
    \varphi(n) = \inf_{K \subset \R^n\text{ convex body}}\sym_{\fold}(K)\, .
\end{equation*}
Note the range of possibilities for $\varphi(n)$ is $[0,1]$. Considering the case of plane convex bodies, Lassak made an initial improvement on this, showing $\varphi(2) \in [1/4,1]$. 

\subsection{Lower Bound} We first improve the lower bound by setting up another \hyperref[foldprog]{optimization} problem. We begin as before with a convex body $K$ and an inscribed axially regular convex hexagon, normalized to have area $1$. Reuse the same labelling of the hexagon and six-pointed star as in Figure~\ref{axialhex}. It will be convenient to discuss our constraints when the hexagon is embedded in a coordinate plane. Let the vertices $A$, $B$, $D$, and $E$ of the hexagon in Figure~\ref{axialhex} have coordinates $(w/2,h)$, $(-w/2,h)$, $(-w/2,-h)$, and $(w/2,-h)$, respectively. Suppose that the midpoint of $\overline{CF}$ is $(w u,0)$, for some $0 \leq u \leq 1/2$. This determines the coordinates of all other vertices in Figure~\ref{axialhex}; we display the resulting coordinates in Figure~\ref{foldhex}. Note that since the hexagon has area $1$, we have $hw = 1/3$. Let $a,b,c,d,e,f,t$ have the same meaning as in Section~\ref{subsec:ax-lb}. The solution to Program~\ref{axialprogram} proved in Appendix~\ref{axialanalysis} gives constraints~\eqref{tFoldCon}. We denote by $\lambda$, the halved folding symmetry of $K$; it is the objective function of our program. 

\begin{figure}[!t]
\begin{center}
\begin{tikzpicture}[scale=1]

\def\x{3};
\def\s{0.5};
\def\h{2.23606};

\draw[thick, purple] plot [smooth cycle] coordinates {(-\x/2,\h) (\x/2,\h) (\x+\s+0.9,\h*0.9,) (\x+\s,0) (\x/2,-\h) (-\x/2,-\h) (-\x+\s+0.3,-\h*0.4,) (-\x+\s,0)};

\draw 
(-\s,-2*\h)--(-1.5*\x+2*\s,\h)
(-\s,-2*\h)--(1.5*\x+2*\s,\h)
(-\s,2*\h)--(-1.5*\x+2*\s,-\h)
(-\s,2*\h)--(1.5*\x+2*\s,-\h)
(1.5*\x+2*\s,-\h)--(-1.5*\x+2*\s,-\h)
(1.5*\x+2*\s,\h)--(-1.5*\x+2*\s,\h);

\filldraw
(\x/2,\h) circle(2pt)
(-\x/2,\h) circle(2pt)
(\x/2,-\h) circle(2pt)
(-\x/2,-\h) circle(2pt)
(-\x+\s,0) circle(2pt)
(\x+\s,0) circle(2pt)
(-1.5*\x+2*\s,\h) circle(2pt)
(-1.5*\x+2*\s,-\h) circle(2pt)
(1.5*\x+2*\s,\h) circle(2pt)
(1.5*\x+2*\s,-\h) circle(2pt)
(-\s,2*\h) circle(2pt)
(-\s,-2*\h) circle(2pt)
(0,0) circle(2pt)
(\s,0) circle(2pt);

\draw
(\x/2+0.3,\h+0.4) node{\tiny $(\frac{w}{2},h)$}
(\x/2,\h-0.25) node{\tiny \color{blue}$A$}

(-\x/2-0.5,\h+0.4) node{\tiny$(-\frac{w}{2},h)$}
(-\x/2,\h-0.25) node{\tiny \color{blue}$B$}

(\x/2+0.45,-\h-0.4) node{\tiny$(\frac{w}{2},-h)$}
(\x/2,-\h+0.25) node{\tiny \color{blue}$E$}

(-\x/2-0.65,-\h-0.4) node{\tiny$(-\frac{w}{2},-h)$}
(-\x/2+0.1,-\h+0.25) node{\tiny \color{blue}$D$}

(-\x-1+\s,0) node{\tiny$(w(u-1),0)$}
(-\x-1+\s+1.3,0) node{\tiny \color{blue}$C$}

(\x+\s+1,0) node{\tiny$(w(u+1),0)$}
(\x-1+\s+0.7,0) node{\tiny \color{blue}$F$}

(-1.5*\x+2*\s-1.25,\h+0.3) node{\tiny$(w(2u-\frac{3}{2}),h)$}
(-1.5*\x+2*\s-1.25,\h) node{\tiny \color{blue}$C'$}

(-1.5*\x+2*\s-1.2,-\h-0.3) node{\tiny$(w(2u-\frac{3}{2}),-h)$}
(-1.5*\x+2*\s-1.25,-\h+0.1) node{\tiny \color{blue}$D'$}

(1.5*\x+2*\s+1.25,\h+0.3) node{\tiny$(w(2u+\frac{3}{2}),h)$}
(1.5*\x+2*\s+1.25,\h) node{\tiny \color{blue}$A'$}

(1.5*\x+2*\s+1.25,-\h-0.3) node{\tiny$(w(2u+\frac{3}{2}),-h)$}
(1.5*\x+2*\s+1.25,-\h+0.1)  node{\tiny \color{blue}$F'$}

(-\s,2*\h+0.3) node{\tiny$(-w u, 2h)$}
(-\s+0.05,2*\h-0.3)  node{\tiny \color{blue}$B'$}

(-\s,-2*\h-0.3) node{\tiny$(-w u, -2h)$}
(-\s+0.1,-2*\h+0.3)  node{\tiny \color{blue}$E'$}

(-0.5,\h-0.25) node{\tiny \color{blue}$P$}
(0.5,-\h+0.25) node{\tiny \color{blue}$Q$}

(\x+\s-0.25,0.865*\h-0.25) node{\tiny \color{blue}$X$}
(-\x+\s+0.25,-0.64*\h+0.25) node{\tiny \color{blue}$Y$}

(4.2,4.25) node{\tiny$x= \alpha w$}
(-2.6,4.25) node{\tiny$x= \beta w$}

(-0.5,0) node{\tiny$(0,0)$}
(\s+0.6,0) node{\tiny$(w u,0)$}
(5.4,0.74) node{\tiny$(\alpha w,\frac{2\alpha-2u-3}{2u+1}h)$}
(0.5-0.6,-\h-0.4) node{\tiny$(m_2 w, -h-\phi_E h)$}
(-0.5+0.35,\h+0.4) node{\tiny$(m_1 w, h+\phi_B h)$}
(\x+\s+0.63,0.865*\h+0.13) node{\tiny$(w(u+1), k_1h)$}
(-\x+\s+1.65,-0.64*\h+0.25) node{\tiny$(w(u-1), -k_2h)$};

\draw[dotted]
(-\x/2,\h)--(\x/2,-\h)
(\x/2,\h)--(-\x/2,-\h)
(4.13,4)--(4.13,-4)
(\x/2,\h)--(4.13,0.75)
(\x+\s,3)--(\x+\s,-3)
(-2.69,4)--(-2.69,-4)
(-\x+\s,3)--(-\x+\s,-3)
(-\x-2*\s, \h+2.5*0.165*\h)--(2*\x,\h-2.25*0.165*\h);

\draw
(4.13,0.7) circle(2pt);

\filldraw
(0.5,-\h-0.17) circle(2pt)
(-0.5,\h+0.16) circle(2pt)
(\x+\s,\h*0.835) circle(2pt)
(-\x+\s,-\h*0.64) circle(2pt);

\end{tikzpicture}
\end{center}
\caption{Inscribed axially regular hexagon with coordinates for the folding program.}
\label{foldhex}
\end{figure}

Suppose the rightmost and leftmost part of $K$ have $x$-coordinates $\alpha w$ and $\beta w$, respectively; this gives \eqref{alphaBetaFoldCon}. We will construct folds along vertical lines that capture as much area as possible from the right and left side of $K$. Let $P$ and $Q$ be the topmost and bottommost points of $K$; suppose their coordinates are $(m_1 w, (\phi_B+1) h)$ and $(m_2 w, -(\phi_E+1) h)$, respectively. 

\begin{lemma}
    The area in $K$ to the right of $x = v_1 w$ can be folded into the area on the left, where
    \begin{equation*}
       v_1 = \max \left\{\frac{2m_1+1}{4},\frac{2m_2+1}{4},\frac{2\alpha-1}{4}\right\}\, . 
    \end{equation*}
\end{lemma}

\begin{proof}
The average of $\alpha w$ and $-w/2$ is $w(2\alpha - 1)/4$. Hence the area in $K$ between $y = -h$ and $y = h$ belonging to the fold will be folded inside of $K$ because $v_1 \geq (2\alpha - 1)/4$. The boundary of $K$ between $B$ and $A$ is decreasing from $x = m_1w$ to $x = w/2$. Hence a fold along $x = (m_1w + w/2)/2 = w(2m_1 + 1) /4$ would fold area inside $K \cap \Delta ABB'$ into area inside $K \cap \Delta ABB'$. A symmetric situation occurs in $\Delta DEE'$.
\end{proof}

Larger $\alpha$ and smaller $\beta$ imply better positive lower bounds on $a+f$ and $c+d$, we calculate these next. Without loss of generality, suppose the line $x = \alpha w$ meets the boundary of $K$ inside $\Delta AFA'$. Area $a$ is made smallest when the boundary of $K$ meets line $x = \alpha w$ on segment $\overline{FA'}$, the coordinates of this point are $(\alpha w, (2\alpha - 2u-3)h/(2u+1))$. It follows that 
\begin{equation*}
    a \geq \frac{1}{2}(2u+1)hw - \frac{1}{2}hw (2u+1) \bigg( \frac{4u-2\alpha+3}{2u+1}\bigg) = \frac{\alpha-u-1}{3}\, .
\end{equation*}
Similarly, if we suppose that $x = \beta w$ meets the boundary of $K$ inside $\Delta BCC'$, then we obtain $c \geq (-\beta + u - 1)/3$. This proves constraints \eqref{sideAreaFoldCon}. 

It will be helpful to assume $v_1,v_2 \in [-1/2,1/2]$. Note that since $\beta \geq u-1 \geq -1$, we immediately obtain $(2\beta + 1)/4 \geq -1/4$ and so $v_2 \in [-1/2,1/2]$. On the other hand, suppose $\alpha > 3/2$. Then the area inside $K$ to the right of the folding line $x = (\alpha/2-1/4)w$ is at least
\begin{equation*}
    hw\bigg(u+1-\frac{1}{2}\bigg) - 2hw\bigg(\alpha/2-1/4-\frac{1}{2}\bigg) + \frac{\alpha-u-1}{3} = \frac{1}{3}\, .
\end{equation*}
From the axial symmetry \hyperref[axialprogram]{program}, or constraint \eqref{tFoldCon}, we know that $|K| = 1+t < 1.44$. Therefore we obtain the bound $\lambda > 0.23$ when $\alpha >3/2$. It follows that we may assume constraints \eqref{vMFoldCon}.

Next we calculate the area inside $K$ to the right of the folding line $x = v_1w$; it is the sum of the areas of a rectangle, the triangle $AEF$, and the part of $K$ inside triangles $AFA'$ and $EFF'$. The area is at least
\begin{equation*}
    2h\left( \frac{w}{2} - v_1 w \right) + \left( wu + \frac{w}{2}\right)h + a + f = \frac{1}{2} - \frac{2}{3} v_1 + \frac{1}{3} u + a + f\, .
\end{equation*}
A similar calculation gives a lower bound on the area to the left of $x = v_2 w$. This gives constraints \eqref{rightfold} and \eqref{leftfold}. 

Note that $b$ is less than equal to area of the trapezium formed by the line parallel to AB passing through $P$, and the lines $AB$, $AB'$, and $BB'$; this gives the first part of constraints \eqref{eqn:abtu}. The second part follows similarly. For constraints \eqref{eqn:abtl}, note that $b$ and $e$ are greater than equal to the areas of $\Delta APB$ and $\Delta DQE$, respectively. Constraints \eqref{eqn:st1} follow from the fact that $P$ lies below the lines $AB'$ and $BB'$. Constraints \eqref{eqn:sb1} follow similarly. 

We may assume without loss of generality that the supporting line of $K$ at $F$ makes an acute or right angle from the line $DE$ in the counterclockwise direction. Let
\begin{equation*}
    k_1 = \sup \{k \in [0,1]: (w(u+1), kh) \in \text{int}(K)\}\, ,
\end{equation*}
where we define the supremum of an empty set to be $0$. Let $X$ be the point on $\partial K$ with coordinates $(w(u+1), k_1h)$.

Suppose that the tangent to $\partial K$ at $C$ also makes an acute or right angle from the line $DE$ in the anticlockwise direction. Then we may define a point $Y$ with coordinates $(w(u-1), -k_2h)$ in a manner similar to $X$. Note that $a$ and $d$ are greater than equal to the areas of $\Delta AXF$ and $\Delta CYD$, respectively; this gives constraints \eqref{eqn:arll}. Constraint \eqref{eqn:st} arises from the fact that the absolute slope of $AX$ is larger than that of $AP$. Constraint \eqref{eqn:sb} follows similarly.

Let $y_1 = \max\{k_1, 1/2\}$. We claim that it is possible to fold the convex body from the top along the line $y = y_1 h$. Let the line $y = y_1 h$ intersect $\partial K$ in points $M$ and $N$ on the left and right, respectively. Let $M'$ and $N'$ be the projections of $M$ and $N$ onto the line $CF$, respectively. The conditions $y_1 \ge k_1$ and $y_1 \ge 0.5$ together ensure that the region $MBAN$ (with $MB$ and $AN$ arcs along $\partial K$) folds from the top into the rectangle $MNN'M'$ and the region $BPA$ (with $BP$ and $PA$ arcs along $\partial K$) folds into the rectangle $ABDE$. A similar argument shows that one can fold along the line $y = y_2 h$ from the bottom.

Next, the area of the fold obtained along the line $y = y_1 h$ is at least the sum of the areas of the region $BPA$ (with $BP$ and $PA$ arcs along $\partial K$) and the trapezium cut off from the hexagon $ABCDEF$ by the folding line $y = y_1 h$. This gives constraint \eqref{topfold}. Constraint \eqref{bottomfold} follows similarly.

\phantomsection
\label{foldprog}
\begin{align}
\textsc{Folding symmetry } & \textsc{program.} \nonumber \\
\textsc{Variables: } & \lambda,\,a,\,b,\,c,\,d,\,e,\,f,\,t,\,u,\,m_1,\,m_2,\,v_1,\,v_2,\,\alpha,\,\beta,\,\phi_B,\,\phi_E,\,k_1,\,k_2,\,y_1,\,y_2 \nonumber \\
\textsc{Minimize: } & \lambda  \nonumber \\
\textsc{Subject to: }
& a,\,b,\,c,\,d,\,e,\,f,\,u \in \left[0,\,\frac{1}{2}\right], \nonumber \\
& \alpha \in \left[u+1,\,2u + \frac{3}{2}\right], \quad \beta \in \left[2u-\frac{3}{2},\,u-1\right], \label{alphaBetaFoldCon} \\
&  v_1,\,v_2,\,m_1,\,m_2 \in \left[-\frac{1}{2},\,\frac{1}{2}\right], \label{vMFoldCon} \\
& \phi_B,\, \phi_E,\, k_1,\, k_2 \in [0,1]\,, \quad y_1,\, y_2 \in \left[\frac{1}{2},\, 1\right], \nonumber \\
& t = a+b+c+d+e+f\,, \quad t \le \frac{6-3\sqrt{2}}{4}\,,  \label{tFoldCon} \\
& a + f \geq \frac{\alpha-u-1}{3}\,, \quad c + d \geq \frac{-\beta+u-1}{3}\,, \label{sideAreaFoldCon} \\
& b \le \frac{1}{6} (1 - (1-\phi_B)^2), \quad e \le \frac{1}{6} (1 - (1-\phi_E)^2)\,, \label{eqn:abtu} \\
& b \ge \frac{\phi_B}{6}\,, \quad e \ge \frac{\phi_E}{6}, \label{eqn:abtl} \\
& a \ge \frac{k_1}{12}(2u + 1)\,, \quad d \ge \frac{k_2}{12}(1 - 2u)\,, \label{eqn:arll} \\
& (1-2m_1) \ge \phi_B (1+2u)\,, \quad (2m_1+1) \ge \phi_B (1-2u)\,, \label{eqn:st1} \\
& (1-2m_2) \ge \phi_E (1+2u)\,, \quad (2m_2+1) \ge \phi_E (1-2u)\,, \label{eqn:sb1} \\
& (1-2m_1) (1-k_1) \ge \phi_B (2u+1)\,, \label{eqn:st} \\ 
& (1+2m_2) (1-k_2) \ge \phi_E (1-2u)\,, \label{eqn:sb} \\
& v_1 = \max \left\{\frac{2m_1+1}{4},\,\frac{2m_2+1}{4},\,\frac{2\alpha-1}{4}\right\}\,, \nonumber \\
& v_2 = \min\left\{\frac{2m_1-1}{4},\,\frac{2m_2-1}{4},\,\frac{2\beta+1}{4}\right\}\,, \nonumber \\
& y_1 = \max \{k_1, \,1/2\}, \quad y_2 = \max \{k_2, \,1/2\}\,, \label{eqn:btf} \\
& (1+t) \lambda \geq \frac{1}{2}-\frac{2}{3}v_1+\frac{1}{3}u + a + f\,, \label{rightfold} \\
& (1+t) \lambda \geq \frac{1}{2}+\frac{2}{3}v_2-\frac{1}{3}u + c + d\,, \label{leftfold} \\
& (1+t) \lambda \ge b + \frac{1}{6}(1-y_1)(3-y_1)\,, \label{topfold} \\
& (1+t) \lambda \ge e + \frac{1}{6}(1-y_2)(3-y_2)\,. \label{bottomfold}
\end{align}

\vspace{10pt}
Finally, we consider the case when the tangent to $\partial K$ at $C$ makes an obtuse angle from the line $DE$ in the anticlockwise direction. In this case, the point $Y$ lies on the arc $BC$ with coordinates $(w(u-1), k_2h)$. By similar arguments as the previous case, we get the following constraints in lieu of \eqref{eqn:arll}, \eqref{eqn:sb}, and \eqref{eqn:btf}, respectively.
\begin{align*}
    &a \ge \frac{k_1}{12}(2u + 1)\,, \quad c \ge \frac{k_2}{12}(1 - 2u)\,, \\
    &(1 + 2m_1)(1 - k_2) \ge \phi_B (1 - 2u)\,, \\
    &y_1 = \max \{k_1, \,k_2,\, 1/2\}\,, \quad y_2 = 1/2\,.
\end{align*}

\begin{prop}
    \label{prop:foldpart1}
    $\varphi(2) \geq 3/8$.
\end{prop}

\begin{proof}
    We used Gurobi~\cite{Gurobi} to solve our \hyperref[foldprog]{program}. Although the program is non-convex, it can be written so that it consists only of linear and quadratic constraints, which makes it susceptible to solution using Gurobi's Mixed-Integer Quadratically Constrained Program (MIQCP) solver. As per the previous discussion, we have two different programs based on the slope of the tangent to $\partial K$ at $C$. In both cases, Gurobi's MIQCP solver is able to establish a lower bound of $0.18803$ on $\lambda$.

    To establish these lower bounds, Gurobi first translates the program into bilinear form and then uses a spatial branch and bound algorithm to find feasible dual solutions with high objective values. In our computations, we set the constraint violation tolerance to $10^{-9}$. Since the combined total number of variables and constraints in our program is bounded above by $100$, the resulting error in the lower bound is a few orders of magnitude smaller than $10^{-3}$; hence we get the desired result.
\end{proof}

\subsection{Upper Bound}

Lassak did not provide an upper bound for this problem, but it was claimed in \cite{MN} that a family of parallelograms can be constructed to show $\varphi(2) \leq 1/2$. Unfortunately, this paper contains a critical error, which we will explain in detail during the proof of the following theorem.

\begin{prop}
    \label{prop:foldpart2}
    There is a sequence of parallelograms for which the folding symmetry approaches $1/\phi$, where $\phi$ is the golden ratio. Therefore, $\varphi(2) \le 1/\phi \approx 0.61803$. For parallelograms, this value is tight, i.e., $\mathrm{Sym}_{\mathrm{fold}}(P) > 1/\phi$ for any parallelogram $P$.
\end{prop}

\begin{proof}
\begin{figure}[H]
    \centering
    \includegraphics[scale=0.55]{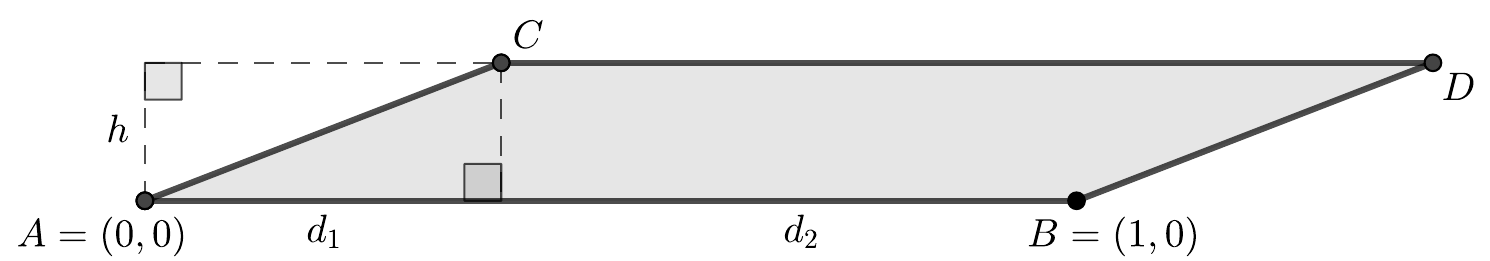}
    \caption{A parallelogram without large folds.}
    \label{fig:folding}
\end{figure}

We consider the possible folds in a parallelogram as labelled in Figure \ref{fig:folding}. For any parallelogram, orienting it such that the longest side is $AB$ will make it so that $d_2>0$. Now, there are three possible candidates for the best fold -- the folding line $\ell$ can cut through
\begin{enumerate}[label=(\arabic*),leftmargin=4\parindent]
    \item $\overline{AB}$ and $\overline{BD}$,
    \item $\overline{AB}$ and $\overline{CD}$, or
    \item $\overline{BD}$ and $\overline{CD}$.
\end{enumerate}

Notably, the second case is not a possible fold unless $d_2\geq d_1$, and this was the location of the error in \cite{MN}. The second case is simply omitted, and it is a very large fold in their final construction which is a family of parallelograms that becomes nearly rectangular. The other two cases are computed in their paper.

\vspace{5pt}
\noindent \textbf{Case 1.}
\begin{figure}[H]
    \centering
    \includegraphics[scale=0.55]{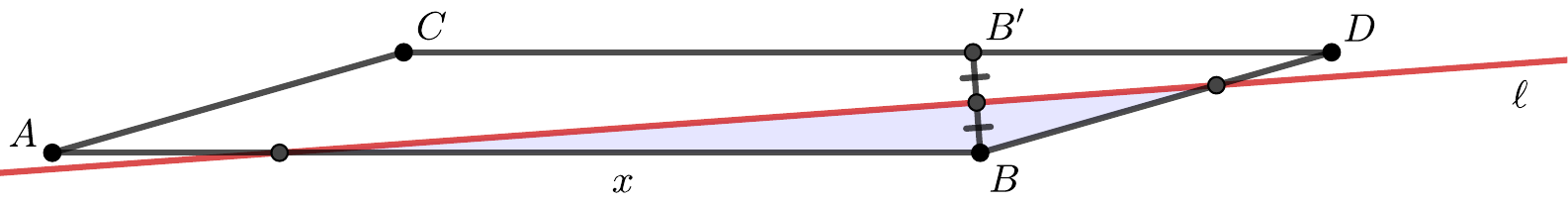}
    \caption{The fold in case 1 (the blue triangular region).}
\label{fig:foldingcase1}
\end{figure}

We consider a line $\ell$ which cuts $\overline{AB}$ at distance $x$ from $B$. Note that if the reflection $B'$ of $B$ in $\ell$ does not lie on $\overline{CD}$, the size of the fold can be increased by translating $\ell$ upwards until it does. It is shown in \cite{MN} that twice the area of the fold in this case (depicted in Figure~\ref{fig:foldingcase1}) is
\begin{equation*}
    2\cdot\brax{\text{Area of fold}} = \frac{hx^2}{x-d_1 + \sqrt{x^2-h^2}}\, .
\end{equation*}
Note that $d_1 < x \leq 1$, otherwise $\ell$ would pass above $A$ or $D$. Now a quick optimization reveals that the size of the fold is maximized at $x=1$. Thus the ratio of twice the area of the maximal fold in this case with the area of the parallelogram is 
\begin{equation*}
    \frac{1}{1-d_1+\sqrt{1-h^2}}\, .
\end{equation*}

\vspace{5pt}
\noindent \textbf{Case 2.} There is no freedom in choosing the angle of $\ell$, so the best fold is along the perpendicular bisector of $\overline{CD}$ (or $\overline{AB}$). This fold yields an area ratio of \begin{equation*}
    \frac{1}{h}\sqrax{\frac{h}{2} (2d_2-2d_1) + \frac{h}{2}(2d_1)} = 1-d_1\, .
\end{equation*}

\vspace{5pt}
\noindent \textbf{Case 3.} 
In this case, twice the area of the maximal fold is found to be
\begin{equation*}
    |\overline{BD}|^2\sin( \angle BDC) = |\overline{BD}|^2\frac{h}{\sqrt{d_1^2 + h^2}} = h\sqrt{d_1^2 + h^2} \, .
\end{equation*}

\vspace{5pt}
So in general, for any parallelogram $P$, we have
\begin{equation*}
    \sym_{\fold}(P) = \max\set{\frac{1}{1-d_1+\sqrt{1-h^2}}, \,\sqrt{d_1^2 + h^2}, 1-d_1}\, .
\end{equation*}
Observe that for a fixed $d_1$, taking $h=0$ minimizes each term, so to find the lowest possible symmetry we minimize $\max\set{\frac{1}{2-d_1}, d_1, 1-d_1}$. This occurs at $d_1=2-\phi$. A family of parallelograms with $d_1=2-\phi$ and $h$ descending to 0 therefore approaches a symmetry of $1/\phi$, and no parallelogram attains this value or anything lower.
\end{proof}

\subsection{Lower bound in the centrally symmetric case}

We had another idea to obtain larger folds, which works for convex sets that are centrally symmetric. We will point out later how this idea might be easily generalizable to all convex bodies. Define $\varphi_{cs}(n)$ to be the minimum folding symmetry, taken only over the centrally symmetric bodies.

\begin{prop}
\label{prop:foldpart3}
    $\varphi_{cs}(2)\geq 4/9$.
\end{prop}

Our proof idea is to consider appropriate folds after obtaining a large inscribed rectangle in the convex body. Radziszewski~\cite{Rad} showed that any convex body has an inscribed rectangle of area at least $1/2$, but as we shall later see, it would be more helpful if we could find smaller inscribed rectangles. In particular, we need the following lemma. If a parallelogram $P$ is inscribed in a convex body $K$, then $\text{int}(K) \setminus P$ consists of $4$ connected components whose closures we call \emph{caps}. 

\begin{lemma}
\label{lem:cs-ins-rect}
   Any centrally symmetric convex body $K$ of area $1$ has an inscribed rectangle of area $r$ for any $r\leq 1/2$. 
\end{lemma}

\begin{proof}
   First, suppose $K$ is strictly convex. Translate so that the origin is the centroid of $K$. For a given point $p \in \partial K$, we create an inscribed parallelogram having area $r$ with one vertex at $p$ as follows. Begin with the unique supporting line $\ell$ at $p$, and rotate it clockwise, causing it to form a chord in $K$. The length of this chord is unimodal, with a maximum at some rotation $\theta$. For any angle between $0$ and $\theta$, there is a unique parallel second chord of equal length, and the four end points of these chords form a parallelogram. 
   
   The area of the parallelogram formed as a function of the rotation over $[0,\theta]$ is continuous, and we claim it has a maximum value of at least $1/2$. To see this, label the cap areas counterclockwise $a$, $b$, $c$, and $d$, where $a$ is the cap bordered by $\ell$. Initially, $a+c=0$ and $b+d=1$, and this is reversed once we rotate $\ell$ to angle $\theta$. Thus, we can rotate until $a+c = 1/4$, and let $r$ be the area of the parallelogram in this case. Through an area preserving affine transformation, we can map the parallelogram to a square with center at the origin. A computation reveals that the amount of area in the other two caps is at most $1/4$, hence the parallelogram has area at least $1/2$. Said computation is almost identical to the one that proves Lemma \ref{lem:cs-large-cap} below, so we omit the details. We remark that one could also use the feature of central symmetry to just apply Lemma \ref{lem:cs-large-cap} towards an easy contradiction, but central symmetry is not necessary for the claim to hold.
   
   This means we can rotate until we first reach a parallelogram of area $r$, which we denote by $P_p$. This process applied to every point on the boundary creates a continuous family of parallelograms. Now, define the functions
   \begin{align*}
       \tau: \partial K \to \partial K,&\ \ p \mapsto \text{ the next vertex after $p$ on $P_p$ clockwise along $\partial K$},\\
       \rho: \partial K \to \mathbb{R},&\ \ p \mapsto \text{the distance of $p$ from the origin, and}\\
       \kappa: \partial K \to \mathbb{R},&\ \ p \mapsto \rho(\tau(p)) - \rho(p).
   \end{align*}
    Note that the above functions are all continuous. The function $\kappa$ effectively measures the `skew' of $P_p$, namely, $\kappa(p) = 0$ if and only if $P_p$ is a rectangle. Therefore, we assume for contradiction that $\kappa(p)$ is never $0$. The function $\kappa$ is continuous, so without loss of generality we can assume it is always positive.
    
    Since $\partial K$ is compact, $\kappa$ attains some minimum value $\delta > 0$ and the function $\rho$ is uniformly continuous. Let $p_0$ be any point on $\partial K$, and let $p_n = \tau(p_{n-1})$ for $n = 1, 2, \dots$. This creates a countable sequence of points on $\partial K$, so we can find two elements arbitrarily close together. In particular, there are distinct elements $p_i$ and $p_j$, such that $\abs{\rho(p_i) - \rho(p_j)} \leq \delta /2$. This is a contradiction, since $\abs{\rho(p_i) - \rho(p_j)} \geq \abs{i-j} \delta > \delta / 2$.

    Extending to bodies that are not strictly convex is a standard limiting argument. For convenience, we replicate the argument found in Section~5 of \cite{Las2}. We can approximate any convex body with a sequence of strictly convex sets $K_i \subseteq K$ converging pointwise to $K$. In each of these, we can inscribe a rectangle $R_i$ of area $r$ from our previous work. By compactness, we can make four selections of finer subsequences to make each of the four vertices converge to a point of $\partial K$.
\end{proof}

Lemma~\ref{lem:cs-ins-rect} is the only gap in proving Theorem \ref{prop:foldpart3} in complete generality. We conjecture that the lemma should be true for all convex bodies, but our method of proof does not seem to generalize. 

Now, if we inscribe a rectangle of area $1/3$ for example, we can immediately find a cap of size at least $1/6$, and this is enough to obtain a fold of size $1/6$. However, using the following lemma which holds for general convex sets, we can improve this further.

\begin{lemma}
\label{lem:cs-large-cap}
    Let $K$ be a convex body of area 1, with an inscribed rectangle $R$ of area $r \le 1/2$. There is a cap of area at least $ \frac{1}{4}(\sqrt{1 - 2 r} + 1 - r)$.
\end{lemma}

\begin{proof}
    There are two opposite caps with areas $a$ and $c$ whose summed area is at least $\frac{1}{2}(1-r)$. Suppose $a\geq c$, and denote the areas of the smaller pair of caps by $b$ and $d$. To start, we instead scale the picture so that the area of the rectangle is $4$, and translate so that the centroid of the rectangle is the origin. We will modify $K$, and attempt to put as much area into the other two caps as possible while maintaining the area of $a+c$ and keeping $R$ fixed.

    \begin{figure}[H]
        \centering
        \includegraphics[scale = 0.55]{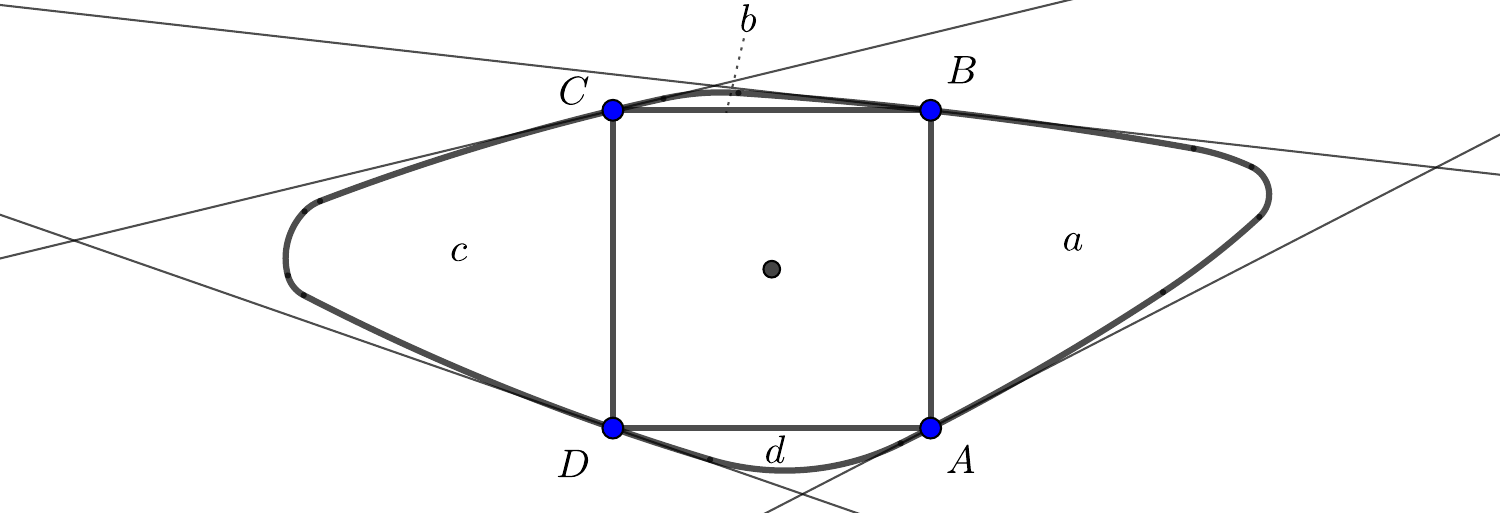}
        \caption{The caps of a convex body with an inscribed rectangle.}
        \label{fig:fold-lower}
    \end{figure}
    
    Let the vertices of the rectangle be $A$, $B$, $C$, and $D$ as shown in Figure~\ref{fig:fold-lower}, where the side cap cut off by $AB$ has area $a$. First, observe that every cap should be made triangular. Indeed, if the smaller caps are not triangular, there is additional area that may freely be added. If the area $a$ cap is not triangular, we can `squish' the cap, by turning supporting lines at $A$, $B$ inward onto the cap until their crossing, along with $A$ and $B$, forms a triangle with area $a$, and replace the cap with this triangle. This creates additional room for the small caps to occupy.    

    Next, let $P$ and $Q$ be the extra vertices of the large triangular caps with areas $a$ and $c$, respectively. These two points now control the areas of all the caps, in the sense that the lines $PA$, $PB$, $QC$, and $QD$ border all of them. Now, by translating $Q$ up and down, we keep $a$ and $c$ constant while changing the areas of the small caps $b$ and $d$. A routine computation of the areas as a function of $Q_y$ shows that $b+d$ is maximized when $Q_y = P_y$. At this stage, we can translate $P$ and $Q$ together, which does not change $a+c$ and $b+d$. Therefore, to minimize the largest cap, we will set $P_y = Q_y = 0$ and $P_x = -Q_x$. 

    Finally, we have reached a centrally symmetric rhombus whose largest cap is simple to compute. For the total area, we quarter the diagram, resulting in a triangle with an inscribed rectangle of area $1$, as shown in Figure~\ref{fig:triangle}. 

    \begin{figure}[H]
        \centering
        \includegraphics[scale = 0.55]{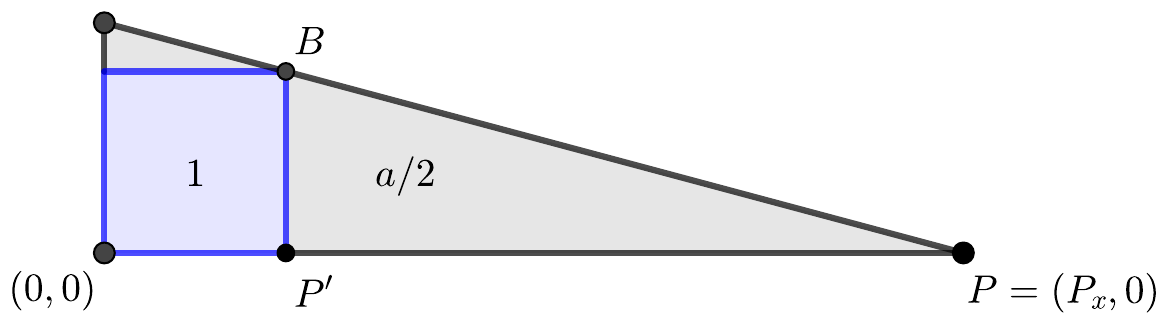}
        \caption{The triangle used to compute the largest cap area.}
    \label{fig:triangle}
    \end{figure}
    
    The ratio between the areas of the rectangle and the triangle is $r' \leq r$, and we compute the area of the cap-portion $\triangle BP'P$ of the triangle to be 
    \begin{equation*}
        \frac{a}{2} = \frac{1}{2r'}(\sqrt{1 - 2 r'} + 1-r') \ge \frac{1}{2r}(\sqrt{1 - 2 r} + 1-r)\, .
    \end{equation*}
    Returning to the original scale yields the desired number. 
\end{proof}

\begin{proof}[Proof of Proposition~\ref{prop:foldpart3}]
    Let $K$ be a centrally symmetric plane convex body with unit area. Find an inscribed rectangle $R$ in $K$ with area $4/9$. Let the vertices of the rectangle be $A$, $B$, $C$, and $D$ as shown in Figure~\ref{fig:fold-lower}. Let $a$, $b$, $c$, and $d$ denote the areas of the caps with bases $\overline{AB}$, $\overline{BC}$, $\overline{CD}$, and $\overline{DA}$, respectively. Without loss of the generality, we may assume that $a$ is the largest of these areas. By Lemma \ref{lem:cs-large-cap},
    \begin{equation*}
        a \ge \frac{1}{4} \brax{\sqrt{1 - 2\brax{\frac{4}{9}}} + 1-\brax{\frac{4}{9}}} = \frac{2}{9}\, .
    \end{equation*}
    Let $S$ be the cap with base $\overline{AB}$ and $P$ be a point in $S$ that is farthest from $AB$. Let $Q$ be the foot of the perpendicular from $P$ to $AB$. If $|\overline{PQ}| \le |\overline{BC}|$, then the cap $S$ can be folded along the line $AB$ to stay fully inside $K$, in which case
    \begin{equation*}
        \text{Sym}_{\text{fold}}(K) \ge 2 \times \text{area}(S) = 2a \ge 4/9\, .
    \end{equation*}
    Thus, we may henceforth assume $|\overline{PQ}| > |\overline{BC}|$. Let $\alpha > 0$ be such that $|\overline{PQ}| = (1+2\alpha) |\overline{BC}|$. Let $\ell$ be the line cutting through $S$ and parallel to $AB$ at a perpendicular distance $(1+\alpha)|\overline{BC}|$ from $P$. Let $S'$ be the smaller cap obtained by the intersection of $S$ with the halfspace defined by $\ell$ containing $P$. It is evident that $S'$ can be folded along $\ell$ to stay fully inside $K$. Let $X$ and $Y$ be the intersection points of $\ell$ with the lines $BC$ and $DA$, respectively. Then we have
    \begin{align*}
        \text{area}(S') &\ge \text{area of triangle } APB - \text{area of rectangle } ABXY, \\
        &= \frac{1}{2} \cdot |\overline{AB}| \cdot (1+2\alpha)|\overline{BC}| - |\overline{AB}| \cdot \alpha|\overline{BC}|\, , \\
        &= \frac{1}{2} |\overline{AB}| |\overline{BC}| = \frac{2}{9}\, .
    \end{align*}
    This implies that $\text{Sym}_{\text{fold}}(K) \ge 4/9$.
\end{proof}

We remark that $4/9$ is the optimal choice for the area of the inscribed rectangle in the above proof. This can be verified by taking an inscribed rectangle of area $r$ and finding the $r$ for which the resulting expression for the lower bound on $\text{Sym}_{\text{fold}}(K)$ is maximized.

Proposition~\ref{prop:foldpart3} together with Propositions~\ref{prop:foldpart1} and \ref{prop:foldpart2} yields Theorem~\ref{thm:fold}.

\section{Bounds in higher dimensions}
\label{sec:high}

In dimensions $3$ and above, we know of no existing upper bounds on $\sigma(n,k)$ for $k>0$ in the literature. In this section we obtain a bound on $\sigma(n+1,k+1)$ in terms of $\sigma(n,k)$, which in particular yields the best known upper bounds on $\sigma(n,n-1)$ for all $n \geq 3$. We need the following lemma, which helps bound the symmetry of bodies that are very near other bodies in volume.

\begin{lemma}
\label{lem:wiggle}
    Let $K \subset \R^n$ be a convex body, and let $k \in \set{0, \dots, n-1}$. Define two other convex bodies, $K_1 \subseteq K \subseteq K_2$ such that $\vol_n(K_1) \ge (1-\ve) \vol_n(K_2)$ for a positive constant $\ve < 1$, and $\sym_k(K_1) = \sym_k(K_2)$. Then $|\sym_k(K)-\sym_k(K_2)| \leq \ve$.
\end{lemma}

\begin{proof}
Let $\mathcal{L}$ be the maximal $k$-flat so that $\vol_n\brax{K \cap \refl_\mathcal{L}(K)} = \vol_n(K) \sym_k(K)$. $K \cap \refl_\mathcal{L}(K)$ is contained in $K_2 \cap \refl_\mathcal{L}(K_2)$, so 
\begin{equation*}
    \sym_k(K_2) \ge \frac{1}{\vol_n(K_2)} \vol_n(K) \sym_k(K) \ge (1-\ve) \sym_k(K)\, .
\end{equation*}
On the other hand, let $\mathcal{L}'$ be the $k$-flat so that $\vol_n \brax{K_1 \cap \refl_{\mathcal{L}'}(K_1)} = \vol_n(K_1) \sym_k(K_1)$. $K_1 \cap \refl_{\mathcal{L}'}(K_1)$ is contained in $K \cap \refl_{\mathcal{L}'}(K)$, so  
\begin{equation*}
    \sym_k(K) \ge \frac{1}{\vol_n(K)} \vol_n(K_1)\sym_k(K_1) \ge (1-\ve) \sym_k(K_1) = (1-\ve) \sym_k(K_2)\, ,
\end{equation*}
and the claim follows.
\end{proof}

In the following, we use $\circc(K)$ to denote the diameter of the smallest closed sphere fully containing a bounded set $K$.

\begin{proof}[Proof of Theorem~\ref{thm:higherdimensions}]
Let $\ve>0$, then take a body $K_{n}\in \R^n$ with $k$-symmetry $\le \sigma(n,k) + \ve$ and re-scale so that $\vol_n(K_n) = n+1$. Also, translate so the centroid of $K_n$ is the origin. Create a pyramid $K_{n+1}$ in $\R^{n+1}$ by placing a point $p$ orthogonally above the centroid of $K_n$ at a suitable distance $D$ which will be specified later (but for now, assume it is larger than $2 \cdot (n+1) \cdot \circc(K_n)$). Denote by $\ell_D$ the line joining $p$ to the centroid of $K_n$ as shown in figure \ref{fig:pyramid}. Observe that the volume of this new shape is $\vol_{n+1}(K_{n+1})= \frac{D}{n+1}\vol_n(K_n) = D$.

\begin{figure}[H]
    \centering
    \includegraphics[scale=0.6]{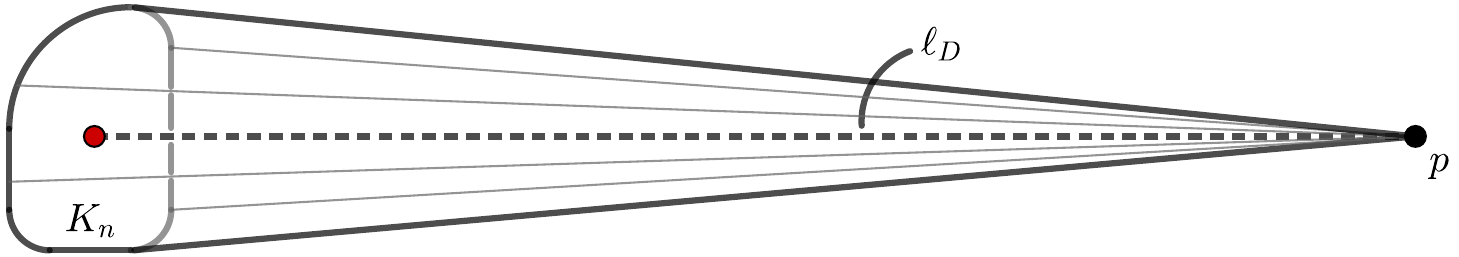}
    \caption{The pyramid construction, $K_{n+1}$.}
    \label{fig:pyramid}
\end{figure}

We now define $\theta_D$ to be the small angle in the right triangle with non-hypotenuse side lengths $D$ and $2(n+1) \circc(K_n)$. 
Given a $k+1$ dimensional subspace $\mathcal{L}\subseteq\R^{n+1}$, there are three cases for the angle between $\mathcal{L}$ and $\ell_D$:
\begin{enumerate}[label=(\arabic*)]
    \item $\mathcal{L}$ is within $\theta_D$ of $\ell_D$,
    \item $\mathcal{L}$ is within $\theta_D$ of the orthogonal complement of $\ell_D$, or
    \item neither of the above hold.
\end{enumerate}

\vspace{5pt}
\noindent \textbf{Case 1.}
We will attain a bound on the symmetry realized by such a subspace slightly above $\sym_{n-1}(K_n)$. Fix a hyperplane $\mathcal{L}$ sufficiently close in angle to $\ell_D$. Then the projection $\ell$ of $\ell_D$ onto this subspace is at an angle of $\leq \theta_D$ with $\ell_D$. Parameterize hyperplanes $H_t$ orthogonal to $\ell$, where $H_0$ passes through $p$, and as $t$ increases, we move in the direction towards the base of the pyramid. We will let $\tau = \sup_{t>0} \brax{H_t \cap K_{n+1} \neq \emptyset}$. Finally, we treat $K_t := H_t \cap K_{n+1}$ as a subset of $\R^n$ with the origin at the intersection of $\ell_D$ with $K_t$ and let $f(t) = \vol_n(K_t) \sym_{n-1}(K_t)$. We now estimate that
\begin{equation}
    \label{eq:integ1}
    \brax{k_{n+1} \cap \refl_{\mathcal{L}}(K_{n+1})} \le \int_0^\tau f(t) dt\, ,
\end{equation}
as we suppose that our chosen subspace happened to reflect every slice optimally. Now, increase $D$ until $\theta_D$ becomes small enough so there are parameters $t_1\leq t\leq t_2$ such that $\frac{t_1}{D}\cdot K_n \subseteq K_t \subseteq \frac{t_2}{D} \cdot K_n$ with $t_1$ sufficiently close to $t_2$ to apply Lemma \ref{lem:wiggle} using $\ve$. In particular, one can compute that suitable slices parallel to the base can be chosen at a distance at most $D \sec(\theta_D) - D + 2\tan(\theta_D)$ apart along $\ell_D$, and by plugging $\theta_D$ into this expression we see it goes to zero as $D$ grows. As a by-product, we also have that $\vol_n(K_t) \le \vol_n \brax{\frac{t}{D} K_n} (1+\ve)$. Recall that the volume of $\frac{t}{D} K_n$ is $\brax{\frac{t}{D}}^n \vol_n(K_n)$. We show this configuration of slices in Figure \ref{fig:slices}.

\begin{figure}[H]
\centering
\includegraphics[scale=0.6]{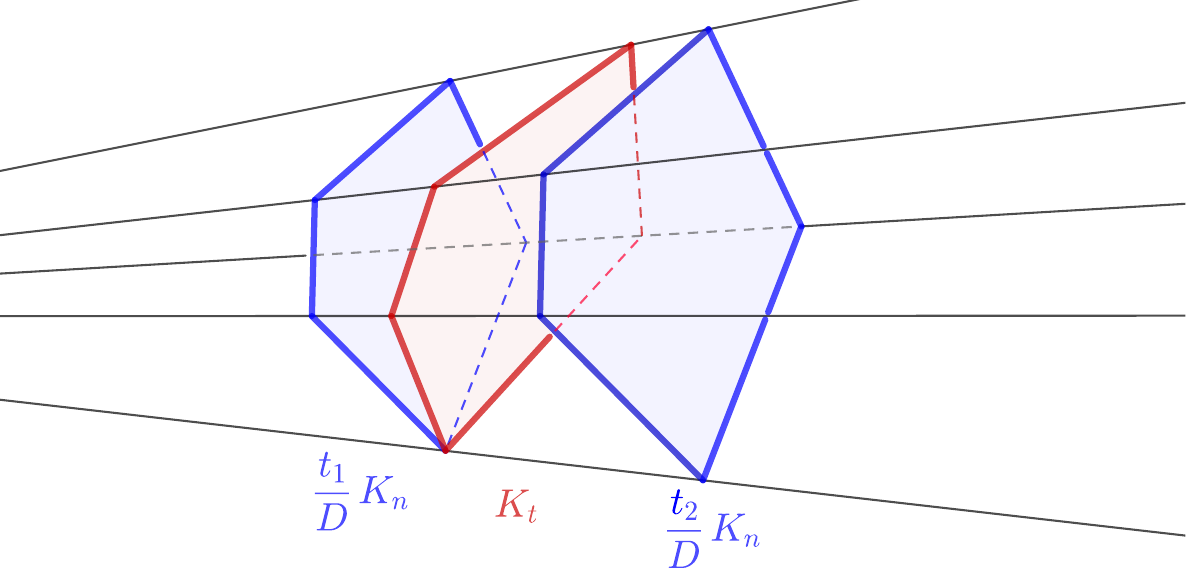}
\caption{The slice $K_t$, between two scaled copies of the base of the pyramid.}
\label{fig:slices}
\end{figure}

Now, it is always the case that $\tau \ge D$, but by choosing $D$ large enough, we can make $\tau^{n+1} \leq D^{n+1} (1+\ve)$. Using this fact with equation \eqref{eq:integ1}, we have
\begin{equation*}
\begin{split}
    \frac{1}{\vol_{n+1}(K_{n+1})}&\vol_{n+1}\brax{K_{n+1}\cap \refl_{\mathcal{L}}(K_{n+1})} 
    \\
    &\leq \frac{1}{\vol_{n+1}(K_{n+1})}\int_0^\tau \brax{\vol_n\brax{\frac{t}{D}K_n}(1+\ve)}\brax{(\sym_{n-1}(K_n) + \ve } dt \, ,
    \\
    &= 
    (1+\ve)\frac{\frac{\tau^{n+1}}{D^{n}(n+1)}\vol_n(K_n)(\sym_{n-1}(K_n)+\ve)}{\vol_{n+1}(K_{n+1})} \, ,
    \\
    &\leq 
    (1+\ve)\frac{D^{n+1}(1+\ve)}{D^{n+1}}(\sym_{n-1}(K_n)+\ve)\, ,
    \\
    &=
    (1+\ve)^{2}(\sym_{n-1}(K_n) +\ve)\, .
\end{split}
\end{equation*}

\vspace{5pt}
\noindent \textbf{Case 2.}
Take a hyperplane $H$ that contains $\mathcal{L}$, and whose normal vector makes an angle less than $\theta_D$ with $\ell_D$. Let $\tau$ now be the distance from $q$ to the base of the pyramid. 
Let $H_t$ be the hyperplanes parallel to $H$, parameterized so that $H_0$ passes through $p$, and $H_{D-\tau}$ passes through $q$, and let $K_t = H_t\cap K_{n+1}$. A reflection of any slice $K_t$ in $\mathcal{L}$ can be decomposed into a reflection in $H$ and a reflection within $H_t$.

Taking $D$ to be larger if necessary, we make $\vol_n(K_t) \le \vol_n\brax{\frac{t}{D} K_n}(1+\ve)$. Supposing that the reflections of each slice intersect perfectly, we have the following upper bound,
\begin{equation*}
\begin{split}
    \frac{ \vol_{n+1}\brax{K_{n+1}\cap \refl_{\mathcal{L}}(K_{n+1})}}{ \vol_{n+1}\brax{K_{n+1}}}
    &\leq \frac{2}{D}\int_{D-2\tau}^{D-\tau} \vol_n\brax{\frac{t}{D} K_n}(1+\ve) dt\, ,
    \\
    &=
    \frac{2}{D}\int_{D-2\tau}^{D-\tau} \frac{t^n}{D^n}(n+1))(1+\ve) dt\, ,
    \\
    &=
    \frac{2(1+\ve)}{D^{n+1}} \sqrax{(D-\tau)^{n+1}-(D-2\tau)^{n+1}} =: *\, ,
\end{split}
\end{equation*}
which attains its max at $\tau = D\brax{ \frac{1}{2 - 2^{1/n + 2}} + \frac{1}{2}} $. Plugging this in and simplifying, we get
\begin{equation*}
\begin{split}
    *&\leq 2(1+\ve) \sqrax{\brax{\frac{{2}^{1/n}}{2^{1/n + 1}-1}}^{n+1}-\brax{\frac{1}{2^{1/n + 1}-1}}^{n+1}}\, ,
    \\
    &=
    (1+\ve) \brax{\frac{1}{2-2^{-1/n}}}^{n}\, ,
\end{split}
\end{equation*}
which creates the second part of the bound in the theorem.

\vspace{5pt}
\noindent \textbf{Case 3.}
Here we can overestimate how much overlap there is by considering a cylinder of diameter $\circc(K_n)$ and height $D$. Because of the skew in angle, the reflection of the cylinder can overlap with the original cylinder in only at most $\circc(K_n) \cot(\theta_D) = \frac{D}{2(n+1)}$ of its height. Thus, the maximal volume of overlap is $\leq \vol_n(K_n) \cdot \frac{D}{2(n+1)} = \frac{D}{2}$, i.e. at most $\frac{1}{2}$ the area of the pyramid. This bound is subsumed by the one in Case 2 for all $n$.

\vspace{5pt}
Altogether, in all cases $\max\set{(1+\ve)\brax{2-2^{-1/n}}^{-n},(1+\ve)^2(\sigma(n,n-1) +2\ve)}$ is the most symmetry a subspace can yield. Since $\ve$ was arbitrary, we are done.
\end{proof}

\appendix

\section{Analysis of the program for Proposition~\ref{prop:ax-lower}}
\label{axialanalysis}

We will consider four cases according to inequalities between $a,f$ and $b,e$. By symmetry, we may assume $c \leq d$ without loss of generality.

\vspace{5pt}
\noindent \textbf{Case 1: $a \leq f, b \geq e$.} Sum $3$ times constraint \eqref{addax} and $3/2$ times \eqref{tribdf} and obtain
\begin{equation}
\label{case1axial}
    \frac{9}{2}(1+t) \lambda \geq 3+\frac{3}{2}(3\sqrt{2}-4) + 3t\, .
\end{equation}
The inequality above becomes weaker on $\lambda$ as $t$ grows. By constraints \eqref{ABDax} and \eqref{CEFax}, we see $t \leq 2/3$. Substituting $t = 2/3$ into \eqref{case1axial} gives the bound $\lambda > 0.715.$

\vspace{5pt}
\noindent \textbf{Case 2: $a \geq f, b \geq e$.} By constraint \eqref{addax},
\begin{equation*}
    \lambda \geq \frac{1+f+e+c}{1+t} \geq \frac{1}{1+a+b+d}\, .
\end{equation*}
By constraint \eqref{ABDax}, the above is at least $3/4$.

\vspace{5pt}
\noindent \textbf{Case 3: $a \geq f, b \leq e$.} The bound in Theorem~\ref{thm:axiality} is tight in this case and the next one. Due to the inequalities between $a, f$ and $b, e$, we only need to consider the following sub-program of the full axial symmetry \hyperref[axialprogram]{program}.

\begin{align*}
\textsc{Variables: } & \lambda,\,a,\,b,\,c,\,d,\,e,\,f,\,t \nonumber \\
\textsc{Minimize: } & \lambda  \nonumber \\
\textsc{Subject to: } & a,\,b,\,c,\,d,\,e,\,f \geq 0\,, \nonumber \\
& a+b+c+d+e+f \geq t\,, \nonumber \\
& d +e \leq 1/6\,, \nonumber \\
& a+ b+d \leq 1/3\,, \nonumber \\
& (1+t) \lambda \geq 1+f+b+c\,, \nonumber \\
& (1+t) \lambda \geq  3\sqrt{2} - 4 + 2(a+c+e)\,.
\end{align*}

\vspace{10pt}
Our strategy will be to treat $t$ as a constant, so that the above is a linear program. We will write down the corresponding dual program, and give a lower bound on the maximum of the dual for all values of $t$. The dual program is

\begin{align*}
\textsc{Variables: }& y_1, \,y_2,\,y_3,\,y_4,\,y_5 \nonumber \\
\textsc{Maximize: } & ty_1-y_2/6-y_3/3+y_4+(3\sqrt{2}-4)y_5 \nonumber \\
\textsc{Subject to: } & y_1,\,y_2,\,y_3,\,y_4,\,y_5 \geq 0\,, \nonumber \\
& (1+t)(y_4+y_5) \leq 1\,, \nonumber \\
& y_1-y_3-2y_5 \leq 0\,, \nonumber \\
& y_1-y_3-y_4-y_5 \leq 0\,, \nonumber \\
& y_1-y_4-2y_5 \leq 0\,, \nonumber \\
& y_1-y_2-y_3 \leq 0\,, \nonumber \\
& y_1-y_2-2y_5 \leq 0\,, \nonumber \\
& y_1-y_4 \leq 0\,. \nonumber
\end{align*}

\vspace{10pt}
The dual variables correspond to the equations of the primal in the order written down above. We observe two feasible solutions to the dual:
\begin{equation*}
   (y_1,\,y_2,\,y_3,\,y_4,\,y_5) = \frac{1}{1+t}(0,0,0,1,0), \quad \text{and} \quad (y_1,\,y_2,\,y_3,\,y_4,\,y_5) = \frac{1}{1+t}\bigg(\frac{4}{5},\,\frac{2}{5},\,\frac{2}{5},\,\frac{4}{5},\,\frac{1}{5}\bigg)\, . 
\end{equation*}
This shows that the objective of the modified primal program above satisfies
\begin{equation*}
   \lambda \geq \max\left\{\frac{1}{1+t}, \frac{1}{5(1+t)}\big(4t+3\sqrt{2}-1\big) \right\}\, . 
\end{equation*}
The minimum value is achieved at $t = (6-3\sqrt{2})/4$, giving the bound $\lambda \geq 4/(10-3\sqrt{2})$.

\vspace{5pt}
\noindent \textbf{Case 4: $a \leq f, b \leq e$.} Similar to the previous case, we only need to consider the following sub-program of the full axial symmetry \hyperref[axialprogram]{program}.

\begin{align*}
\textsc{Variables: } & \lambda,\,a,\,b,\,c,\,d,\,e,\,f,\,t \nonumber \\
\textsc{Minimize: } & \lambda \nonumber \\
\textsc{Subject to: } & a,\,b,\,c,\,d,\,e,\,f \geq 0\,, \nonumber \\
&  a+b+c+d+e+f\geq t\,, \nonumber \\
& d +e \leq 1/6\,, \\
& c+e +f \leq 1/3\,, \\
& (1+t) \lambda \geq 1+a+b+c\,, \\
& (1+t) \lambda \geq 3 \sqrt{2}-4 + 2(b+d+f)\,.
\end{align*}

\vspace{10pt}
By swapping variables $a$ with $f$, $b$ with $c$, and $d$ with $e$, we obtain exactly the primal sub-program written down in Case~3.

\vspace{5pt}
This concludes the proof of Proposition~\ref{prop:ax-lower}. 

\section{A detailed proof of Proposition~\ref{prop:r2part2}}
\label{app:axcomp}

In this appendix, we prove $\sigma(2,1)\leq \frac{1}{3}(\sqrt{2} + 1)$ by analyzing the axiality of the quadrilateral $\set{(0,0), (1,0), \left(1, \ve\frac{\sqrt{2}}{1 + \sqrt{2}}\right), \left(\frac{1}{\sqrt{2}},\ve\right)}$ shown in Figure~\ref{fig:2Dupperbound}. By sending $\ve\to 0$, the axiality approaches the desired number. 

Let $\alpha$ be the argument of a direction vector normal to the reflection line, which we denote $\ell$. The vertices of the reflected body are called $A'$, $B'$, $C'$, and $D'$. To summarize each case in the proof, we break up the choices for the angle and translate of $\ell$ into the following:
\begin{enumerate}[label=(\arabic*),leftmargin=4\parindent,itemsep=4pt]
    \item \textbf{Small angle case}: $0\leq \alpha \leq \frac{1}{2}\arctan\brax{\frac{\sqrt{2}\ve}{1+\sqrt{2}}}$. The translate is recorded using the $x$-intercept of $\ell$, and can take values such that
    \begin{enumerate}[itemsep=2pt]
        \item $\ell$ intersects $\overline{AC}$, or
        \item $\ell$ intersects $\overline{CD}$.
    \end{enumerate}
    \item \textbf{Middle angle case}: $\frac{\pi}{2} \leq \alpha  \leq \frac{1}{2}(\pi + \arctan(\sqrt{2}\ve))$. The translate is recorded using the $y$-intercept of $\ell$, and may take values that are
    \begin{enumerate}[itemsep=2pt]
        \item positive and large enough so that $C'$ is above $AB$,
        \item negative,
        \item positive such that $\ell$ intersects $\overline{CD}$ and $C'$ is below $AB$, or
        \item positive such that $\ell$ intersects $\overline{BD}$ and $C'$ is below $AB$.
    \end{enumerate}
    \item $-\frac{\pi}{4} \leq \alpha < \frac{\pi}{4}$, excluding the interval of the `Small angle case'. All translates are easily ruled out.
    \item  $\frac{1}{2}(\pi + \arctan(\sqrt{2}\ve)) \leq \alpha < \frac{3\pi}{4}$. We check translates such that
    \begin{enumerate}[itemsep=2pt]
        \item $C'$ lies above $AB$, or
        \item $C'$ lies below $AB$. 
    \end{enumerate}
    \item $\frac{\pi}{4} \leq \alpha < \frac{\pi}{2}$. All translates are easily ruled out.
\end{enumerate}

\subsection{Cases (1) and (3)}

We first observe that reflection lines passing through $\overline{CD}$ are quite poor, so we can rule out case (1.b). Indeed, as $\ve$ becomes small, such a line yields symmetry at most $\frac{1}{2} (3 - \sqrt{2})\approx 0.79$, since it loses the entire triangle cut off by the line $D'B'$. Further, an angle for which $B'$ is above $AB$, or $D'$ is below $AC$, is sharply worse than otherwise. $B'$ lying below $AB$ implies $\alpha \geq 0$, and $D'$ lying above $AC$ implies $\alpha \leq \frac{1}{2}\arctan\brax{\frac{\sqrt{2}\ve}{1+\sqrt{2}}}$, which is because the largest possible such angle comes from the perpendicular bisector of $A$ and $D$. We have therefore ruled out case (3), and it only remains to check in this section the case (1.a), which is shown in Figure \ref{fig:uppercase1a}.

With these considerations, let $R$ and $S$ be the intersections of $\overline{D'B'}$ with $\overline{AB}$ and $\overline{AC}$, respectively, and $P$ and $Q$ the intersections of $\ell$ with $\overline{AB}$ and $\overline{AC}$, respectively. Let $t$ be the distance from $A$ to $P$.

\begin{figure}[H]
    \centering
    \includegraphics[scale=0.53]{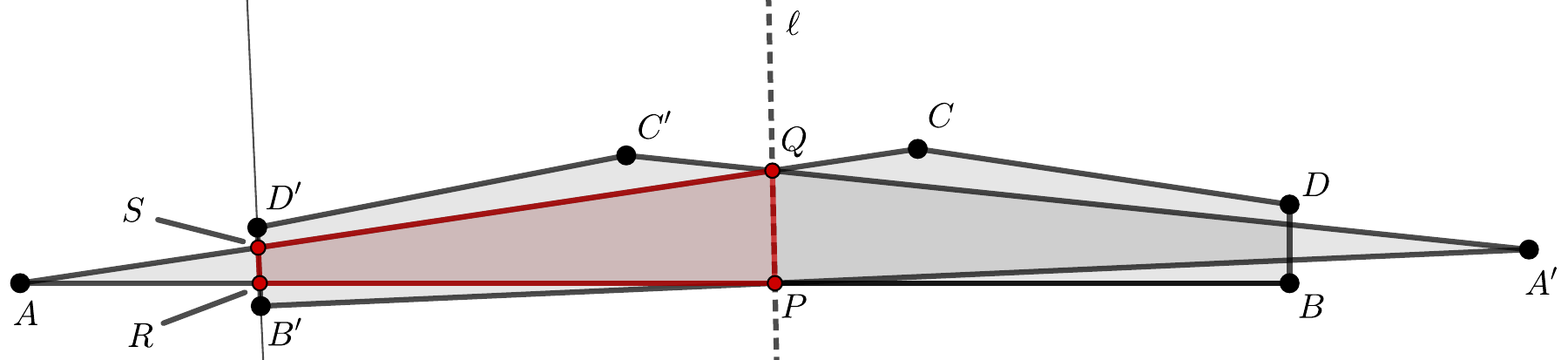}
    \caption{Case (1.a).}
    \label{fig:uppercase1a}
\end{figure}

We compute the locations of the four points, 
\begin{equation*}
\begin{split}
    P &= t(1,0)\, ,
    \\
    R &= \brax{t-\frac{1-t}{\cos(2\alpha)}}(1,0)\, ,
    \\
    S &= \frac{\cot(2\alpha)}{\sqrt{2}\ve+\cot(2\alpha)}\left(t-\frac{1-t}{\cos(2\alpha)}\right)(1,\sqrt{2}\ve)\, ,
    \\
    Q &= \frac{t\cot(\alpha)}{\sqrt{2}\ve+\cot(\alpha)} (1,\sqrt{2}\ve)\, ,
\end{split}
\end{equation*}
and then use the shoelace formula
to compute the area.
\begin{equation*}
\begin{split}
    2\cdot \text{Area} &= Q_xS_y - Q_yS_x + S_xR_y-S_yR_x + R_xP_y - R_yP_x + P_xQ_y - P_yQ_x\, , \\        
    &= -S_yR_x + P_xQ_y \, , \\
    &=\ve \sqrt{2}\sqrax{
    t^2\brax{\frac{\cot(\alpha)}{\sqrt{2}\ve+\cot(\alpha)}}
    -
    \brax{t-\frac{1-t}{\cos(2\alpha)}}^2\brax{\frac{\cot(2\alpha)}{\sqrt{2}\ve+\cot(2\alpha)}}
    }\, .
\end{split}
\end{equation*}
Expanding this and writing it as a quadratic in $t$ gives 
\begin{equation*}
    \begin{split}
    \frac{2\cdot \text{Area}}{\sqrt{2} \ve} = &\  
    t^2 \brax{\frac{\cot(\alpha)}{\cot(\alpha) + \sqrt{2}\ve}-\frac{\cot(2\alpha)+2\csc(2\alpha)+\csc(2\alpha)\sec(2\alpha)}{\cot(2\alpha) + \sqrt{2}\ve}} 
    \\ & 
    + t\brax{\frac{2\csc(2\alpha) + 2\csc(2\alpha)\sec(2\alpha)}{\cot(2\alpha) + \sqrt{2}\ve}} 
    - \frac{\csc(2\alpha)\sec(2\alpha)}{\cot(2\alpha)+ \sqrt{2}\ve}\, .
    \end{split}
\end{equation*}
We can find the maximum of a quadratic easily, and it simplifies significantly. The maximum area, altogether, is
\begin{equation*}
   \frac{\ve \sqrt{2} }{\sqrt{2}\ve\sin(\alpha) + 2\cos(\alpha)+1} \, .
\end{equation*}
So we now have the maximal ratio with the original body,
\begin{equation*}
    \brax{\frac{1}{2-\sqrt{2}}} \frac{\sqrt{2}}{\sqrt{2}\ve\sin(\alpha) + 2\cos(\alpha)+1} = \frac{1+\sqrt{2}}{\sqrt{2}\ve\sin(\alpha) + 2\cos(\alpha)+1}\, .
\end{equation*}
We need to find for what interval containing $\alpha$ is this ratio at most $\frac{1}{3}(1+\sqrt{2})$. This is equivalent to finding where $\sqrt{2}\ve\sin(\alpha) + 2\cos(\alpha) \geq 2$, which is the case when $0 \leq \alpha < 2 \arctan(\frac{\ve}{\sqrt{2}})$. This in particular is valid for $\alpha \leq \frac{1}{2}\arctan\brax{\frac{\sqrt{2}\ve}{1+\sqrt{2}}}$, so case (1.a) is complete. As an aside, one can easily check that for $\alpha = 0$, the optimal translate of the line occurs when $t=\frac{2}{3}$ and attains the value $\frac{1}{3}(1+\sqrt{2})$ precisely.

\subsection{Cases (2), (4), and (5)}

Here, the translate $t$ is the distance of $\ell$ along the $y$-axis. For convenience, we let $\beta = \alpha - \frac{\pi}{2}$. Denote the internal angle bisector of $AB$ and $AC$ by $\ell_0$, and the angle it makes with $AB$ by $\beta_0 = \frac{1}{2} \arctan(\ve \sqrt{2})$. First, suppose $\beta > \beta_0$, so that we are in case (4). Let $S$ and $S_0$ be the intersections of $\ell$ and $\ell_0$ with $\overline{CD}$ respectively (it will become clear later that $\ell$ must intersect $\overline{CD}$ for any good translate, so $S$ is well defined), and let $C'$ and $C_0'$ be the reflections of $C$ in $\ell$ and $\ell_0$ respectively. We denote the quadrilateral $\brax{A,C_0', S_0, C}$ by $\mathcal{A}_0$, which comes from reflection in $\ell_0$.
\begin{figure}[H]
    \centering
    \includegraphics[scale=0.53]{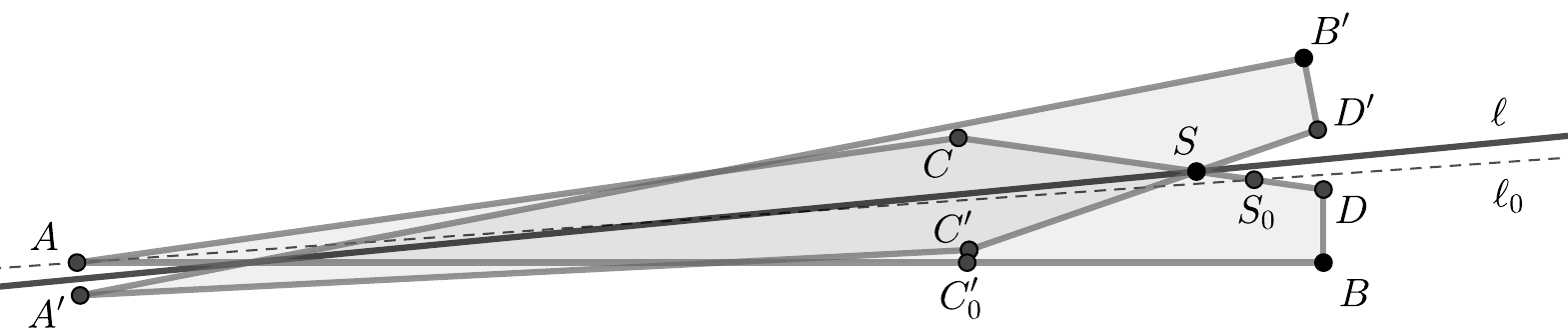}
    \caption{Case (4.a).}
    \label{fig:uppercase4a}
\end{figure}
We first discuss case (4.a), which are reflections like the one in Figure \ref{fig:uppercase4a}. If $C'$ is above $C_0'S_0$ then $\ell$ is trivially worse than $\ell_0$, since the area of overlap obtained by reflection in $\ell$ is strictly contained in $\mathcal{A}_0$. If $C'$ is below $C_0'S_0$ but still above $AB$, we have two triangles $\triangle C_0'C'M$ and $\triangle SS_0M$, where $M=C'S\cap C_0'S_0$. Observe that $\triangle C_0'C'M$ represents an area gained over $\mathcal{A}_0$, and $\triangle SS_0M$ an area lost. We show that $\abs{\triangle C_0'C'M} < \abs{\triangle SS_0M}$, by showing that $M$ approaches $C_0'$ as $\ve$ goes to $0$.

The lines $C'S$ and $C_0'S_0$ are given by the equations 
\begin{equation*}
    y = \frac{2\ve}{1-2\ve^2}(\sqrt{2}x-\sqrt{1+2\ve^2})\text{  and  }
    y = \frac{\sqrt{2} \ve \cot(2 \beta) + 1}{\cot(2 \beta) - \sqrt{2} \ve}x- \frac{(2\ve-t)\csc(2\beta)}{\cot(2 \beta) - \sqrt{2} \ve } + t\, ,
\end{equation*}
respectively. The worst possible translate (giving the largest $\triangle C_0'C'M$ and smallest $\triangle SS_0M$) is when $C'$ lies on $\overline{AB}$. For fixed $\beta$, the appropriate translation required for $C'$ to lie on the $\overline{AB}$ is given by $t=\ve(1+  \frac{1}{2}  \sec^2(\beta)) - \frac{\tan(\beta)}{\sqrt{2}}$. Substituting in $t$ into equations $C_0'S_0$ and $C'S$ then solving simultaneously for $M_x$, we get
\begin{equation*}
     M_x = \frac{(\cot(2 \beta)-\sqrt{2} \ve) (2 \ve \sqrt{1 + 2 \ve^2} + (\ve(1+  \frac{\sec^2(\beta)}{2}  ) - \frac{\tan(\beta)}{\sqrt{2}})(1- 2 \ve^2) ) - \frac{(1-2 \ve^2) ( \frac{\tan(\beta)}{\sqrt{2}}-\frac{\ve\sec^2(\beta)}{2}  )}{\sin(2 \beta)}  }{(1 + 2 \ve^2) ( \sqrt{2} \ve \cot(2 \beta)-1)}\, .
\end{equation*}
As $\ve$ goes to $0$, one can compute that $M_x$ approaches $\frac{1}{\sqrt{2}}$.

\begin{figure}[H]
    \centering
    \includegraphics[scale=0.54]{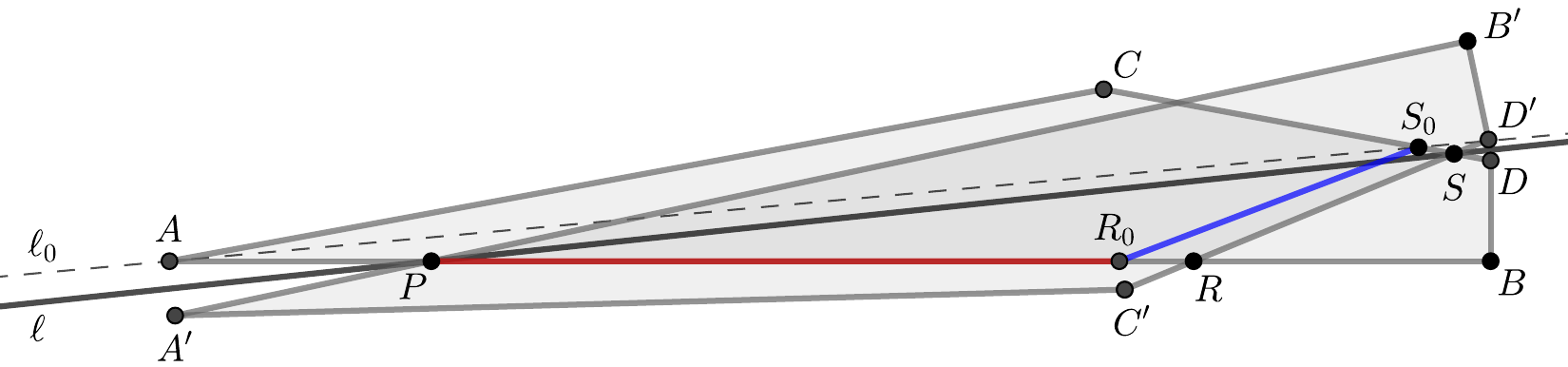}
    \caption{Case (4.b).}
    \label{fig:uppercase4b}
\end{figure}

Next we address case (4.b), shown in Figure \ref{fig:uppercase4b}. We can see that $\overline{PR_0}$ is longer than $\overline{R_0S_0}$, otherwise $\triangle APC$ is easily shown to be too large an area excluded from the overlap. As we translate downwards in this case, we lose a slice of width roughly $|\overline{PR}|\geq |\overline{PR_0}|$, and gain a slice of size roughly $|\overline{RS}|\leq |\overline{R_0S_0}|$. Therefore, one can use an integral argument to see that translation in this direction is a loss, and this case is strictly worse than case (4.a). This concludes all of case (4). 

We move on now to case (2). The hardest out of all the cases to analyze is (2.c), which is depicted in Figure~\ref{fig:uppercase2p2}. 

\begin{figure}[H]
    \centering
    \includegraphics[scale=0.54]{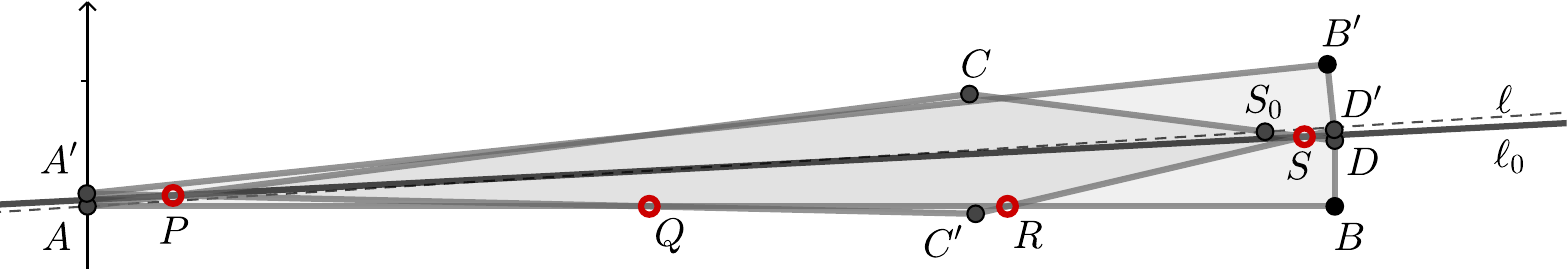}
    \caption{Case (2.c).}
    \label{fig:uppercase2p2}
\end{figure}

We rule out case (2.a), where $C'$ is below $AB$, since translating $\ell$ downwards strictly increases the area. We observe further that by the same argument that disqualifies case (4.b) above, the translate must be positive, so $A'$ sits above $A$, thereby settling case (2.b). We now make the assumption that $\overline{CD}$ and $\overline{C'D'}$ intersect at the point we will call $S$, so we are in case (2.c). Observe that $S$ is still to the right of $S_0$, otherwise the overlap region is entirely contained $\mathcal{A}_0$. We compute the locations of the four vertices in the figure, which are
\begin{equation*}
    \begin{split}
    P &= \frac{t}{\sqrt{2}\ve-\tan(\beta)}(1,\sqrt{2}\ve)\, , \\
    Q&=\brax{P_x\frac{2 \ve^2 + 2 \sqrt{2} \ve \cot(2 \beta) - 1}{(\sqrt{2} \ve \cot(2 \beta) - 1)},\,0}\, ,\\
    R&=\brax{\frac{(2 \ve - t )\sec(2\beta)+ \sqrt{2} \ve t \tan(2 \beta)- t }{\sqrt{2} \ve  + \tan(2 \beta)},\,0}\, ,\\
    S&=\brax{\frac{2\ve-t}{\tan(\beta)+\sqrt{2}\ve},\,2\ve-\sqrt{2}\ve\frac{2\ve-t}{\tan(\beta)+\sqrt{2}\ve}}\, .
    \end{split}
\end{equation*}

This time, there are fewer cancellations in the shoelace formula. We have
\begin{equation*}
\begin{split}
    2\cdot \text{Area} &= P_xQ_y + Q_xR_y + R_xS_y + S_xP_y -Q_xP_y - R_xQ_y - S_xR_y - P_xS_y\, , \\
    &= P_x\cdot 0 + Q_x\cdot 0 + R_x\cdot(-\sqrt{2}\ve S_x+2\ve) + S_x\cdot\sqrt{2}\ve P_x \\
    &\hspace{2cm}-Q_x\cdot\sqrt{2}\ve P_x - R_x\cdot 0 - S_x\cdot 0 - P_x\cdot(-\sqrt{2}\ve S_x+2\ve )\, , \\
    &=\sqrt{2}\ve [\sqrt{2}P_x(\sqrt{2}S_x-1)+R_x(\sqrt{2}-S_x)-P_xQ_x]\,.
\end{split}
\end{equation*}
Plugging those quantities in results in a monstrous expression, which can be expanded and written as a quadratic $at^2+bt+c$ in the same manor as in case (2.c). For simplicity, let $k=\sqrt{2}\ve$.  We computed the coefficients, which are
\begin{equation*}
    \begin{split}
    a =&\, (\tan(\beta ) + k) (\tan(2\beta ) + k) (2 k + (k^2-1) \tan(2\beta ) )
    \\
    & +2 (\tan(\beta ) - k) (\tan(2\beta ) - k) (\tan(2\beta ) + k)
    \\
    & + (\tan(\beta ) - k)^2 (\tan(2\beta ) - k) (k \tan(2\beta ) - \sec(2\beta ) - 1)\, ,
    \\
    b =&\, \sqrt{2}(\tan(\beta)-k)^2(\tan(2\beta)-k)\big((\tan(2\beta)+k)
    \\
    &+ (\tan(\beta)(k \tan(2\beta)-\sec(2\beta)-1)+k\sec(2\beta) )\big)\, ,
    \\
    c =&\, 2k(\tan(\beta)-k)^2(\tan(2\beta)-k)\tan(\beta)\sec(2\beta)\, .
    \end{split}
\end{equation*}
Maximizing over $t$ gives the value 
\begin{equation}
    \label{eq:gequation1}
    m(\beta) = 2(1 + \sqrt{2})k\frac{\tan(\beta)\cdot g(\beta)-k\sec(2\beta)(\tan(\beta) - k)^2(\tan(2\beta) - k)}{\cos(2\beta)(\tan(\beta)+k)(\tan(2\beta)+k)g(\beta)}\, ,
\end{equation}
where 
\begin{equation*}
    \begin{split}
    g(\beta) = &\, (\tan(\beta)+k)(\tan(2\beta)+k)\brax{2k+(2k-1)\tan(2\beta)}
    \\
    &+ 2\brax{\tan(\beta)-k}\brax{\tan(2\beta)-k}\brax{\tan(2\beta)+k}
    \\
    &+ \brax{\tan(\beta)-k}^{2}\brax{\tan(2\beta)-k}\brax{k\tan(2\beta)-\sec(2\beta)-1}\, .
    \end{split}
\end{equation*}
This can be moderately simplified, to
\begin{equation*}
\begin{split}
    m(\beta)&=
    -4(1 + \sqrt{2})k \frac{(k + 2) \cos(\beta) + (k - 2) \cos(3 \beta) - 2 \sin(\beta)}
    { h(\beta) }\, ,
\end{split}
\end{equation*}
where 
\begin{equation*}
    \begin{split}
        h(\beta) = &\, - (14 k^2 + 4 k) \cos(\beta) -8 k^2 \cos(3 \beta) +  (4 k - 2 k^2) \cos(5 \beta) + (2 k - 4) \sin(\beta) \\
        &\, + (k^3 - 2 k^2 + 9 k - 2) \sin(3 \beta) + (k^3 - 2 k^2 - k + 2) \sin(5 \beta)\, .
    \end{split}
\end{equation*}

When $k$ is small, the important pieces become the constant terms, which only occur on the bottom: $-4\sin(\beta) - 2\sin(3\beta) + 2\sin(5\beta)$. However, this collection of terms vanishes when $\beta$ is small. In particular, let $\beta\leq \arctan(k)$ (this value is chosen to avoid singularities which can be found using the first expression for $m$, equation \eqref{eq:gequation1}, otherwise any constant times $\arctan(k)$ would have worked as well). Thus we are interested only in the $k$ terms, which are
\begin{equation*}
    \frac{1}{4(1 + \sqrt{2})} m(\beta) \approx \frac{-2\cos(\beta)+2\cos(3\beta)+2\sin(\beta)}{-4\cos(\beta)+4\cos(5\beta) + 2\sin(\beta)+ 9\sin(3\beta)-\sin(5\beta)}\, .
\end{equation*}
Finally, near 0, the cosine terms {all} cancel, and we are left with the contribution of the sine terms which is $\frac{2}{2+9\cdot 3 - 5} = \frac{1}{12}$, so that $m(\beta)\to \frac{1}{3}(1 + \sqrt{2})$.

\vspace{3mm}
Case (2.d) is when $D'$ is reflected below $D$ so that $\overline{CD}$ and $\overline{C'D'}$ do not intersect. This is in fact expected to be the optimal translate for some angles between $0$ and $\arctan\brax{\frac{h\sqrt{2}}{1+\sqrt{2}}}$, however, we can bound it using the previous computation.
\begin{figure}[H]
    \centering
    \includegraphics[scale=0.54]{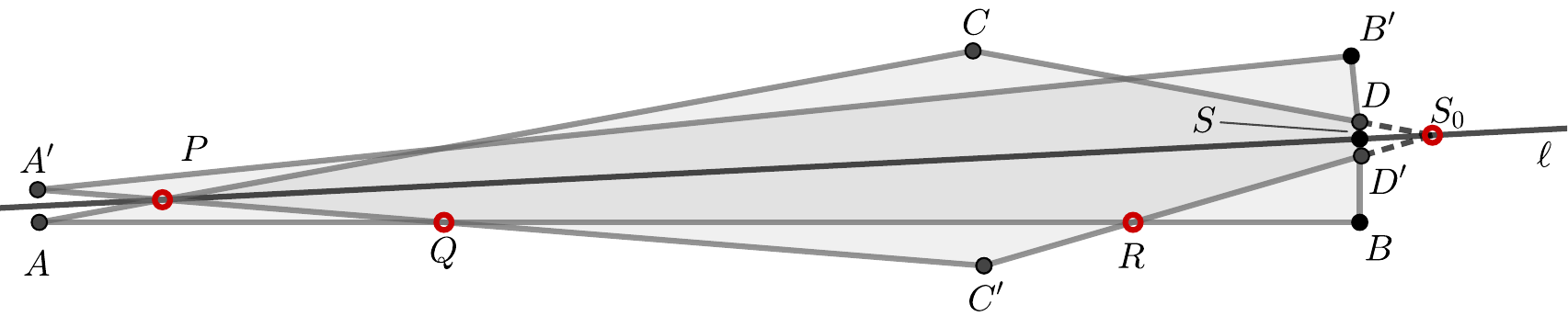}
    \caption{Case (2.d).}
    \label{fig:case2d}
\end{figure}
Our computation in case (2.c) still finds the area of the quadrilateral $PQRS_0$ shown in Figure \ref{fig:case2d}. Since this contains the polygon created from (2.d), we can upper bound it using this area. The ratio from this area approached the correct value for all angles in case (2), so we are done.

Case (5) is easily ruled out as well, noting that the shape is an isosceles triangle with a triangle cut off by the line $BD$. The reflection lines in this case are repetitions of the ones in case (2) and (4), only now, the larger end of the shape is reflected further apart.


\begin{thebibliography}{11}

\bibitem{Ball} K. Ball, Volume ratios and a reverse isoperimetric inequality, \emph{J. London Math. Soc.} \textbf{44}(2) (1991), 351--359. MR1136445

\bibitem{Bes} A. S. Besicovitch, Measure of asymmetry of convex curves, \emph{J. London. Math. Soc.} \textbf{23} (1948), 237--240. MR0027543

\bibitem{BR} A. Bielecki and K. Radziszewski, Sur les parall\'el\'epip\`edes inscrits dans les corps convexes, \emph{Ann. Univ. Mariae Curie-Sk\l{}odowska Sect. A} \textbf{7} (1954), 97--100. MR0081498

\bibitem{BM} A. B. Buda and K. Mislow, On a measure of axiality for triangular domains, \emph{Elem. Math.} \textbf{46}(3) (1991), 65--73. MR1113766

\bibitem{CS} G. D. Chakerian and S. K. Stein, On measures of symmetry of convex bodies, \emph{Canadian J. Math.} \textbf{17} (1965), 497--504. MR0177349

\bibitem{Cho} C.-Y. Choi, Finding the largest inscribed axially symmetric polygon for a convex polygon, Masters thesis, Department of Electrical Engineering and Computer Science, Korea Advanced Institute of Science and Technology, 2006.

\bibitem{deV} B. A. deValcourt, Measures of axial symmetry for ovals, \emph{Israel Journal of Mathematics} \textbf{4}(2) (1966), 65--82. MR0203589

\bibitem{Far} I. F\'ary, Sur la densit\'e des r\'eseaux de domaines convexes, \emph{Bull. Soc. Math. France} \textbf{78} (1950), 152--161. MR0039288

\bibitem{FR} I. F\'ary and L. R\'edei, Der zentralsymmetrische Kern und die zentralsymmetrische H\"ulle von konvexen K\"orpern, \emph{Math. Ann.} \textbf{122} (1950), 205--220. MR0039294

\bibitem{Gil} G. Gilat, On the quantification of asymmetry in nature, \emph{Foundations of Physics Letters} \textbf{3} (1990), 189--196.

\bibitem{GG} G. Gilat and Y. Gordon, Geometric properties of chiral bodies, \emph{J. Math. Chem.} \textbf{16} (1994), 37--48. MR1304180

\bibitem{Gru} B. Gr\"unbaum, Measures of symmetry for convex sets, Proceedings of Symposia in Pure Mathematics, Vol. VII, Convexity, Amer. Math. Soc., Providence, R. I., 1963. MR0156259

\bibitem{Gurobi} Gurobi Optimization, LLC, Gurobi Optimizer Reference Manual, 2023. \url{https://www.gurobi.com}

\bibitem{Kra} F. Krakowski, Bemerkung zu einer Arbeit von W. Nohl, \emph{Elemente der Mathematik} \textbf{18} (1963), 60--61.

\bibitem{Las2} M. Lassak, Approximation of convex bodies by rectangles, \emph{Geom Dedicata} \textbf{47} (1993), 111--117. MR1230108

\bibitem{Las} M. Lassak, Approximation of convex bodies by axially symmetric bodies, \emph{Proc. Amer. Math. Soc.} \textbf{130}(10) (2002), 3075--3084. MR1908932

\bibitem{LL} M. A. Lavrent'ev and L. A. Lyusternik, Elements of the calculus of varlations, Vol. 1, Part II, Moscow, 1935.

\bibitem{program} K. Moore, Convex Symmetry, GitHub Repository, 2023. \url{https://github.com/Kenneth-Moore/Axiality}

\bibitem{Noh} W. Nohl, Die innere axiale Symmetrie zentrischer Eibereiche der euklidischen Ebene, \emph{Elemente der Mathematik} \textbf{17} (1962), 59--63.

\bibitem{MN} M. Nowicka, On the measure of axial symmetry with respect to folding for parallelograms, \emph{Beitr Algebra Geom} \textbf{53} (2012), 97--103. MR2890367

\bibitem{Rad} K. Radziszewski, Sur une probl\`eme extr\'emal relatif aux figures inscrites et circonscrites aux figures convexes, \emph{Ann. Univ. Mariae Curie-Sklodowska, Sect. A} \textbf{6} (1952), 5--18. MR0062461

\bibitem{Ste} S. Stein, The symmetry function in a convex body, \emph{Pacific. J. Math.} \textbf{6} (1956), 145--148. MR0080321

\bibitem{GT} G. Toth, Measures of Symmetry for Convex Sets and Stability, Universitext, Springer, Cham, 2015. MR3410930

\end{thebibliography}
\end{document}